\documentclass{amsart}  
\usepackage{amsmath,amssymb,latexsym}
\usepackage{comment,nicefrac}
\usepackage{tikz,pgfplots}
\usepackage{hyperref,mathabx}
\usetikzlibrary{arrows,snakes,backgrounds}
\usepackage[export]{adjustbox}
\theoremstyle{plain}       
\newtheorem{theorem}{Theorem}[section]
\newtheorem{lemma}[theorem]{Lemma}
\newtheorem{corollary}[theorem]{Corollary}
\newtheorem{example}[theorem]{Example}
\newtheorem{definition}[theorem]{Definition}
\newtheorem{remark}[theorem]{Remark}

\usepackage[normalem]{ulem}
\usepackage{dsfont}
\renewcommand{\triangle}{{T}} 
\newcommand{\KL}{Karhunen-Lo\`{e}ve}

\newcommand{\isdef}{\mathrel{\mathrel{\mathop:}=}}

\renewcommand{\div}{\operatorname{div}}

\newcommand{\balpha}{{\boldsymbol{\alpha}}}
\newcommand{\bbeta}{{\boldsymbol{\beta}}}
\newcommand{\bgamma}{{\boldsymbol{\gamma}}}

\newcommand{\bnu}{{\boldsymbol{\nu}}}
\newcommand{\bxi}{{\boldsymbol{\xi}}}
\newcommand{\brho}{{\boldsymbol{\rho}}}
\newcommand{\bvarrho}{{\boldsymbol{\varrho}}}

\newcommand{\refd}{{\operatorname*{ref}}}
\newcommand{\Dref}{D_{\refd,\kappa}}

\renewcommand{\d}{\operatorname{d}\!}

\newcommand{\norm}[2]{\left\|{#1}\right\|_{#2}}

\newcommand{\brab}[1]{\left\{#1\right\}}
\newcommand{\brac}[1]{\left(#1\right)}

\newcommand{\supp}{\operatorname{supp}}

\def\CC{{\mathbb C}}
\def\CC{{\mathbb C}}

\def\NN{{\mathbb N}}
\def\RR{{\mathbb R}}

\def\FF{{\mathbb F}}

\def\TT{{\mathbb S}^1}
\def\EE{{\mathbb E}}
\def\ZZ{{\mathbb Z}}

\def\RRp{{\mathbb R}_+}

\newcommand{\be}{{\boldsymbol{e}}}

\newcommand{\bk}{{\boldsymbol{k}}}

\newcommand{\bp}{{\boldsymbol{p}}}

\newcommand{\bx}{{\boldsymbol{x}}}

\newcommand{\by}{{\boldsymbol{y}}}

\newcommand{\bz}{{\boldsymbol{z}}}

\newcommand{\bI}{{\boldsymbol{I}}}
\newcommand{\bQ}{{\boldsymbol{Q}}}
\newcommand{\beeta}{{\boldsymbol{\eta}}}

\def\Aa{{\mathcal A}}

\def\Cc{{\mathcal C}}

\def\Jj{{\mathcal J}}

\def\mfu{{\mathfrak u}}
\def\im{{\rm i}}
\def\rd{{\rm d}}
\newcommand{\rev}{{\rm ev}} 

\begin{document}
\title[Sparsity of PDEs on log-Gaussian random shapes]
{Elliptic PDEs on log-Gaussian Shapes:
\\ 
Sparsity and Finite Element Discretization
}

\author{Dinh D\~{u}ng}
\address{Dinh D\~{u}ng,
Information Technology Institute, Vietnam National University,
144 Xuan Thuy, Cau Giay, Hanoi, Vietnam}
\email{dinhzung@gmail.com}
\author{Helmut Harbrecht}
\address{Helmut Harbrecht,
Department of Mathematics and Computer Science,
University of Basel,
Spiegelgasse 1, 4051 Basel, Switzerland}
\email{helmut.harbrecht@unibas.ch}
\author{Van Kien Nguyen}
\address{Van Kien Nguyen,
Department of Mathematical Analysis, 
University of Transport and Communications, No.
3 Cau Giay Street,  
Hanoi, Vietnam}
\email{kiennv@utc.edu.vn}
\author{Christoph Schwab}
\address{Christoph Schwab,
	Seminar for Applied Mathematics, 
	ETH Zurich, 
	R\"amistrasse 101, 
	8092 Zurich, Switzerland}
\email{schwab@math.ethz.ch}

\maketitle
\begin{abstract}
In this article, we consider the solution to elliptic diffusion problems on 
a class of random domains obtained by log-Gaussian random homothety 
of the unit disk respectively an annulus. 
We model the problem under consideration and verify
the existence and uniqueness of the random solution by 
path-wise pullback to the nominal unit disk respectively annulus. 
We prove the analytic regularity of the solution with respect 
to the random input parameter. 
We consider the numerical approximation of the 
random diffusion problem by means of continuous, piecewise linear 
Lagrangian Galerkin Finite Elements 
with numerical quadrature in the nominal domain, 
and by sparse grid interpolation and quadrature of 
Gauss-Hermite Smolyak and Quasi-Monte Carlo type in the parameter domain.
The theoretical findings are complemented by numerical results.
\end{abstract}
%
\section{Introduction}
Domain uncertainties appear in many applications from 
engineering and sciences since the shape of the object 
under consideration may not be perfectly known. For example, 
one might think of manufacturing imperfections in the shape 
of products fabricated by line production, or shapes which 
stem from inverse problems such as tomography. 
Similarly,
the boundary of the object might not be well-defined in the 
sense that it is the result of a smooth transition between 
different materials or states. Here, one may think of 
applications in cell biology for example.

Boundary value problems on random domain have been
considered by several authors, see \cite{CK,CNT,HPS,HSS,
HSSS,Mohan,Xiu} for example. However, in all these articles, 
the domain has been modeled by bounded random variables 
to ensure the random domain under consideration does not 
degenerate. 
In contrast, we model the random domain 
by using log-Gaussian distributed random variables. 
To this end, 
we restrict ourselves to homothetic and star-shaped
random domains where the logarithm of radial function 
is a Fourier series with Gaussian coefficients.
This setup guarantees that almost every
realization of the random domain is 
homothetic to a nominal annulus resp. to the unit disc,
and hence well-defined.

Since the 
realizations of the random domain can exhibit large variations,
we apply the domain mapping approach meaning that the given 
boundary value problem on the random domain is pulled back to the 
unit disc resp. to a nominal annular domain. 
Thus, the boundary value problem posed on the 
random domain becomes a boundary value problem posed on a 
deterministic domain but with log-Gaussian random diffusion matrix 
and right-hand side. This allows to extend the  complex-variable 
methods proposed in \cite[Section 2.2]{CD} and, specifically, in 
\cite[Section 3.6]{DNSZ}, 
using holomorphic extension of the solution to the last equation to 
establish bounds for its parametric partial derivatives in the energy 
norm, 
for the analysis of Sparse Grid and Quasi-Monte Carlo integration.
\subsection{Contents}
This paper is organized as follows. 
We present the model problem and the geometry parametrization in Section~\ref{sct:PDE}. 
The model problem is formulated in Section~\ref{sct:model} and 
upon homothetic pullback on the reference 
domain in Section~\ref{sct:reference}. 
Geometry parametrization and precise analytic dependency on the input parameters 
is investigated in Section~\ref{sct:analyticity}. 
Next, in Section~\ref{sct:randomPDE}, 
we consider the situation that the input parameters are random.
Geometry uncertainty is modelled by 
tensor-product Gaussian measures on the parameter sequences.
The particular log-Gaussian domain model 
is formulated in Section~\eqref{sct:lognormal}. 
Holomorphy of the 
parameters-to-solution map is verified in \ref{sct:holomorphy},
where the precise bounds of the particular partial derivatives 
are derived in Section~\ref{sct:decay}. 
The semidiscretization 
with respect to the random parameter is the topic of Section \ref{sct:numerixSpG}.
Here, we consider the sparse grid 
interpolation of the random solution in Section~\ref{sct:interpolation}
and the quasi-Monte Carlo quadrature based on Halton points 
for directly computing the random solution's expectation in
Section~\ref{sct:QMC}. 
The finite element discretization and 
implemenation of the problem under consideration is the topic
of Section~\ref{sct:numerix}. 
The finite element discretization in the spatial variable
of a generic parameter-dependent boundary value problem
is analysed in Section~\ref{sct:FEM1}, 
while numerical
quadrature is studied in Section~\ref{sct:FEM2}. This theory
is then applied to the problem under consideration in Section
\ref{sct:FEM3}.
Numerical experiments which verify our theoretical findings are
carried out in Section~\ref{sct:experiment}. 
Finally, the conclusion is drawn in Section~\ref{sct:conclusio}.
\subsection{Notation}  
\label{sec:Notat}
We denote the natural numbers by 
$\NN = \{1,2,\ldots\}$, where $\NN_0 := \NN\cup \{0\}$. 
As usual, $\ZZ$ are the integers, $\RR$ the real numbers, 
$\RRp$ the real non-negative numbers, and $\CC$ the complex 
numbers. For a complex number $z \in \CC$, $\mathfrak{R}(z)$ 
and $\mathfrak{I}(z)$ denote the real and imaginary parts of $z$, 
respectively.

Denote by $\RR^\NN$ and $\RR^\ZZ$ the sets of all sequences 
$\by = (y_j)_{j\in \NN}$ and $\by = (y_j)_{j\in \ZZ}$, respectively, with $y_j\in \RR$. 
For a nonnegative sequence $\brho = (\rho_j)_{j\geq 1} \in [0,\infty)^\NN$,
we denote its support $\supp(\brho) = \{ j\in \NN: \rho_j >0 \}$.
With $|\by|_0$ we count the number of nonzero components 
$y_j$ of $\by = (y_j)_{j\in \NN}$. 
The set $\FF$ consists of all sequences 
of non-negative integers $\bnu=(\nu_j)_{j \in \NN}$ 
such that 
$\supp(\bnu) := \{j \in \NN: \nu_j >0\}$ 
is a finite set.  
For $\bnu, \bk \in \FF$ and $k \in \NN$, 
$|\bnu|_1 := \sum_{j \in \NN} \nu_j$,
$|\bnu|_\infty := \max\{\nu_j, j\in \NN\}$, 
and
\[
\bnu^k := \prod_{j \in \supp(\bnu)} \nu_j^k,\qquad 
\bnu^\bk := \prod_{j \in \supp(\bnu)} \nu_j^{k_j}, \qquad
\bnu!:= \prod_{j \in \NN} \nu_j!.
\]

If  $\balpha= (\alpha_j)_{j \in \Jj}$  is a set of positive numbers with any 
index set $\Jj \subseteq \NN$, 
we use the notation $\balpha^{-1}:= (a_j^{-1})_{j \in \Jj}$.
Given the sets 
$\balpha= (a_j)_{j \in \Jj}$ and $\bbeta= (\beta_j)_{j \in \Jj}$,
we define 
$\balpha^\bbeta:= \big(\alpha_j^{\beta_j}\big)_{j \in \Jj}$.

We use the letters $C$  and $K$ to denote general positive constants 
which may take different values, and $C_{\alpha,\beta,\dots}$ and 
$K_{\alpha,\beta,\dots}$ constants depending on $\alpha,\beta,\dots$.
For the quantities $A_n(f,\bk, \by,\ldots)$ and $B_n(f,\bk, \by,\ldots)$ depending on $n,f,\bk,\by,\ldots$,  
we write  $A_n(f,\bk,\by,\ldots) \lesssim B_n(f,\bk,\by,\ldots)$
if there exists a constant $C >0$ independent of $n,f,\bk, \by,\ldots$  such that 
$A_n(f,\bk,\by,\ldots) \le CB_n(f,\bk,\by,\ldots)$.
Finally, we denote by $|G|$ the cardinality of the set $G$.
\section{Elliptic diffusion problems on variable domains}
\label{sct:PDE}
\subsection{A class of homothetic parametric domains}
\label{sct:model}
We shall first model the class of random domains under 
consideration. The model we use is in accordance with 
\cite{HPS}, where however uniformly distributed random 
variables on a compact domain have been considered for 
the input parameters. Since log-Gaussian random domains 
can become unbounded, we restrict ourselves here to 
homothetic domains in order to ensure that the domains 
under consideration are always well-defined. 

To this end, we consider the parametric Poisson equation
\begin{equation}\label{eq:PoissionD_a}
	-\Delta u(a) = f\ \text{in $D_\kappa(a)$},\quad
	u(a) = 0 \ \text{on $\partial D_\kappa(a)$},
\end{equation}
where $f \in L^2(\RR^2)$ and 
$\kappa \in [0,1)$ and $a$ describe 
the shape of the domain $D_\kappa(a) \subset \RR^2$ in polar coordinates 
according to the homothety $r\mapsto a(\theta) r$, i.e.,
\begin{equation}\label{D_a}
D_\kappa(a) :=  
    {\rm int} \brab{\bx = (r \cos \theta, r \sin \theta) \in \RR^2: \ 0 \leq \kappa a(\theta) \leq  r < a(\theta) }.
\end{equation}
Here, ``${\rm int}$'' denotes the interior of a set. 
In \eqref{D_a}, we take $a \in   W^1_\infty(\TT)$, 
understood to be positive, and $2\pi$-periodic,
where $\TT$ is the unit circle and $W^1_\infty(\TT)$
is the space of $2\pi$-periodic, real-valued, 
Lipschitz-continuous functions on $\TT$,
equipped with the norm
\[
	\norm{a}{W^1_\infty(\TT)} :=  \norm{a}{L^\infty(\TT)} +  \norm{a'}{L^\infty(\TT)}.
\]
When 
$\norm{a}{{W^1_\infty(\TT)}}$ 
is finite and $a(\theta) > 0, \ \theta \in \TT$, 
the domain $D_0(a)$ is a bounded domain.
In particular, 
the domain $D_\kappa(a)$ is Lipschitz since the boundary
admits an global representation with respect to the polar angle 
as graph of the Lipschitz smooth radial function. 
Thus, there 
exists a unique weak solution $u(a) \in H^1_0\big({D_\kappa(a)}\big)$ to 
\eqref{eq:PoissionD_a}, satisfying 
the variational formulation
\begin{equation}\label{eq:varformD_a}
	\int_{D_\kappa(a)} \nabla u(a) \cdot\nabla v  \ = \ \langle f,v\rangle, 
		\quad v \in H^1_0\big({D_\kappa(a)}\big).
\end{equation} 
We remark that for $0<\kappa<1$, $\partial D_\kappa(a)$ 
has two components, one of them with arbitrary large curvature as $\kappa\downarrow 0$.

\subsection{Pullback of the solution to a reference domain}
\label{sct:reference}
In order to study the parameter dependent problem 
\eqref{eq:PoissionD_a} and \eqref{D_a}, we transform
it to a parameter independent, 
fixed \emph{reference domain} $\Dref$.
Throughout this article, we consider 
for $0\leq \kappa < 1$ the reference domain
\begin{equation}\label{eq:Dref}
	\Dref := {\rm int} \big\{\bxi\in\mathbb{R}^2: \kappa \le \|\bxi\|_2<1\big\}
	\;.
\end{equation}
For $\kappa = 0$, $\Dref$ is the open unit disc, 
and $\Dref$ and $D_\kappa(a)$ are star-shaped with respect to the origin.
For $0 < \kappa < 1$, $\Dref$ is an annulus.

Given $a$ and $\kappa \in [0,1)$,
let $F(a)$ be the homothetic transform 
which maps  $\Dref$ onto $D_\kappa(a)$ 
that is defined by
\begin{equation}\label{eq:dom_map}
F(a)(\bxi):= F(a)(r \cos \theta, r \sin \theta)
:= 
\big(a(\theta)r \cos \theta, a(\theta)r \sin \theta\big)
\end{equation}
for $\bxi = (r \cos \theta, r \sin \theta) \in \Dref$. 
The Jacobian matrix of $F(a)(\bxi )$ is given by 
\begin{align*}
	\frac{\d F(a)}{\d\bxi}(\bxi) &= a(\theta)
	\begin{bmatrix} \cos \theta & -\sin \theta \\ \sin \theta & \cos \theta\end{bmatrix}
	\begin{bmatrix}
		1 &  h(\theta) \\
		0 &  1 \end{bmatrix} 
	\begin{bmatrix} \cos \theta & \sin \theta \\ -\sin \theta & \cos \theta\end{bmatrix}
\end{align*}	
with
\begin{equation}
	\label{h(theta)}
	h(\theta):= \frac{a'(\theta)}{a(\theta)},
\end{equation}
where $a'(\theta)$ denotes the derivative with respect to variable $\theta$.
The determinant of the Jacobian is
\begin{equation}\label{J_a}
J(a)(\bxi) = \det\bigg(\frac{\d F(a)}{\d\bxi}(\bxi)\bigg) = a(\theta)^2  > 0,
\end{equation} 
which implies that the map $F(a)(\bxi)$ is one-to-one.

We denote the pullback solution of \eqref{eq:PoissionD_a} by
\begin{equation}\label{eq:u(a)} 
	\hat{u}(a)  = u(a) \circ F(a) \in V, \;\; \mbox{where} \;\; V:= H^1_0(\Dref)\;.
\end{equation}
Utilizing $F(a)$ as a change of variable in \eqref{eq:varformD_a}, 
$\hat{u}$ solves
\begin{equation}\label{eq:varformD}
\hat{u}(a) \in V: \quad 
	\int_{\Dref} M(a) \nabla \hat{u}(a) \cdot\nabla v  \ 
	= \ \langle f_{\refd}(a),v\rangle \ \ \forall v \in   V.
\end{equation} 
This is the variational formulation of the equation
\begin{equation}\label{eq:problemD}
	-\div\big(M(a)\nabla \hat{u}(a)\big) 
	= f_{\refd}(a)\ \text{in $\Dref$},\quad
	\hat{u}(a) = 0 \ \text{on $\partial \Dref$}
\end{equation}
(for $0<\kappa<1$, $\partial\Dref = \kappa \TT \cup \TT$ and for $\kappa = 0$, $\partial \Dref = \TT$) 
 with the diffusion matrix
\begin{equation}\label{M_a}
\begin{aligned}
	M(a)(\bxi ) :=& J(a)(\bxi)\bigg(\frac{\d F(a)}{\d\bxi}(\bxi)\bigg)^{-1}
		\bigg(\frac{\d F(a)}{\d\bxi}(\bxi)\bigg)^{-\intercal} \\
	=& \begin{bmatrix} \cos \theta & -\sin \theta \\ \sin \theta & \cos \theta\end{bmatrix}
	\begin{bmatrix}
		1 + h(\theta)^2 & - h(\theta)\\
		- h(\theta) &  1
	\end{bmatrix}
	\begin{bmatrix} \cos \theta & \sin \theta \\ -\sin \theta & \cos \theta\end{bmatrix}
\end{aligned}
\end{equation}
and the right-hand side
\begin{equation}\label{f_a}
	f_{\refd}(a)(\bxi):= J(a)(\bxi)(f\circ F(a)(\bxi)).
\end{equation}

Note that the equation \eqref{eq:PoissionD_a} on the random 
domain $D_\kappa(a)$ and the equation \eqref{eq:problemD} on the 
deterministic domain $\Dref$ are in different Cartesian coordinate 
systems: \eqref{eq:PoissionD_a} is in $\bx$-coordinates, while 
\eqref{eq:problemD} is in $\bxi$-coordinates. The variables $\bx$ 
and $\bxi$ are connected by the equation $\bx= F(a)(\bxi)$ of
the change of variables. For simplicity, $\bx$ and $\bxi$ are 
omitted in the sequel in some equations and expressions.
 
\subsection{Parametric analyticity of the pullback solution}
\label{sct:analyticity}
Throughout the rest of this article, we assume 
\begin{equation}\label{eq:fAn}
f \; \mbox{is a fixed, real analytic function on}\; \RR^2
\;
\mbox{and}
\; 
a\in W^1_\infty(\TT;\CC)\;.
\end{equation}
We define the complex extension of the weak solution as 
\begin{equation}\label{eq: general}
\hat{u}(a) \in {V_\CC}: \quad 
	B\big(\hat{u}(a),v;a\big) = L(v;a) \quad \forall v\in {V_\CC}\;,
\end{equation} 
where $V_\CC = H^1_0(D_{\refd,\kappa};\CC)$ and with $\bar{v}$ denoting the complex conjugate for $v\in V_\CC$. 
\begin{equation}\label{BL}
	B\big(\hat{u}(a),v;a\big)=\int_{\Dref} M(a) \nabla \hat{u}(a) \cdot\nabla \bar v  
	\quad\text{and}\quad
	L(v;a) =  \langle f_{\refd}(a), \bar v\rangle.
\end{equation} 

Due to the expressions of $M(a)$ in \eqref{M_a} and of $J(a)$ 
in \eqref{J_a}, and the assumption that $f$ is real analytic, the 
mapping $a \mapsto L(\cdot;a)$ is holomorphic from $W^1_\infty(\TT;\CC)$ to $V_\CC'$, 
and $a \mapsto B(\cdot,\cdot;a)$ is linear, hence holomorphic, 
from $L^\infty(D_{\refd,\kappa};\CC)$ to $\mathfrak{B}(V_\CC , V_\CC)$, 
the \emph{bilinear} forms  $V_\CC \times V_\CC \to \CC$. 
It is also holomorphic from 
$W^1_\infty(\TT;\CC) \subset L^\infty(\TT;\CC)$ to $\mathfrak{B}(V_\CC , V_\CC)$.

Let the real symmetric matrix $R(a)(\bxi)$ be defined by 
$$
R(a)(\bxi) := 
\mathfrak{R}\big(M(a)(\bxi)\big),
$$
and let 
$\lambda_{\operatorname{min}} {(a)(\bxi)}$ 
resp.~$\lambda_{\operatorname{max}} (a)(\bxi)$ 
denote its smallest resp.~largest eigenvalue. 
Denote
\[
\lambda_{\min}(a) := \min_{\bxi\in \Dref} \lambda_{\operatorname{min}}(a)(\bxi),
\quad
\lambda_{\max}(a) := \max_{\bxi\in \Dref} \lambda_{\operatorname{max}}(a)(\bxi).
\]
Then, we have
\begin{equation}\label{BL-Re}
	\mathfrak{R}\big(B(v,v;a)\big) 
        = 
        \int_{\Dref} R(a) \nabla v \cdot\nabla \bar v
	\ge 
	\lambda_{\operatorname{min}}(a) \norm{v}{V_\CC}^2
\end{equation}
for any $v \in V_\CC$. 
Therefore, the coercivity condition 
\begin{equation}\label{BL-CorCond}
	|B(v,v;a)|\ge \lambda_{\operatorname{min}}(a) \norm{v}{{V_\CC}}^2
\end{equation}
holds.
If  $\lambda_{\operatorname{min}}(a) > 0$, 
by the Lax-Milgram lemma, we obtain
\begin{equation} \label{norm{hat{u}(a)}{V}1}
	\norm{\hat{u}(a)}{{V_\CC}} 
	\le 
\lambda_{\min}(a)^{-1} \norm{f_{\refd}(a)}{L^2(\Dref;{\CC})}.
\end{equation} 
Next, we derive a bound on $\lambda_{\min}(a)^{-1}$.

\begin{lemma}
Assume that $a\in W^1_\infty(\TT;\CC)$ and 
let $h(\theta)$ be defined in \eqref{h(theta)}.
Assume also that
	\[
	\|\mathfrak{I}(h(\theta))\|_{L^\infty(\TT)} < 1.
	\]
Then there holds
\begin{equation} \label{lambda_min}
\lambda_{\min}(a)^{-1} 
\leq 
\frac{2+\|\mathfrak{R}(h(\theta))\|_{L^\infty(\TT)}^2}{1-\|\mathfrak{I}(h(\theta))\|_{L^\infty(\TT)}^2}
\end{equation} 
and
\begin{equation} \label{lambda_max}
	\lambda_{\max}(a)
	\leq 2+\|\mathfrak{R}(h(\theta))\|_{L^\infty(\TT)}^2.
\end{equation} 
\end{lemma}

\begin{proof}
The eigenvalues of $R(a)(\bxi)$ are the roots of the quadratic equation 
\[
	\lambda^2 - B\lambda + C
\]
and, with $\mathfrak{R}(h)^2$ denoting $\mathfrak{R}(h^2)$, that
\[
  C:= 1 - \mathfrak{I}(h(\theta))^2, \quad 
	B:= 2 + \mathfrak{R}(h(\theta))^2 - \mathfrak{I}(h(\theta))^2.
\]
Therefore, 
\begin{align*}
		\lambda_{\operatorname{min}}{(a)(\bxi)}
		& =\frac{2 + \mathfrak{R}(h(\theta))^2 - \mathfrak{I}(h(\theta))^2
		-\sqrt{(\mathfrak{R}(h(\theta))^2 -\mathfrak{I}(h(\theta))^2 )^2 +4 \mathfrak{R}(h(\theta))^2 }}{2}
		\\
		& =
		\frac{2-2\mathfrak{I}(h(\theta))}{2 + \mathfrak{R}(h(\theta))^2 
		- \mathfrak{I}(h(\theta))^2+\sqrt{(\mathfrak{R}(h(\theta))^2 -\mathfrak{I}(h(\theta))^2 )^2 +4 \mathfrak{R}(h(\theta))^2 }}  .
\end{align*}
Then we have the estimate
\begin{equation*}
	\lambda_{\operatorname{min}}(a)(\bxi)^{-1}
		\leq \frac{2 + \mathfrak{R}(h(\theta))^2   
                + \sqrt{ \mathfrak{R}(h(\theta))^4 
		+ 4 \mathfrak{R}(h(\theta))^2 }}{2-2\mathfrak{I}(h(\theta))} 
		\leq \frac{2+\mathfrak{R}(h(\theta))^2}{1-\mathfrak{I}(h(\theta))^2},
\end{equation*}
which implies
\begin{equation*} \label{lambda_min_2}
	\lambda_{\min}(a)^{-1} 
         \leq \frac{2+\|\mathfrak{R}(h(\theta))\|_{L^\infty(\TT)}^2}
	           {1-\|\mathfrak{I}(h(\theta))\|_{L^\infty(\TT)}^2}.
\end{equation*} 

The bound \eqref{lambda_max} is derived in complete analogy.
\end{proof}

From \eqref{norm{hat{u}(a)}{V}1} and \eqref{lambda_min}, 
we obtain
\begin{equation} \label{norm{hat{u}(a)}{V}}
	\norm{\hat{u}(a)}{V_\CC} \le \frac{2+\|\mathfrak{R}(h(\theta))\|_{L^\infty(\TT)}^2}
	{1-\|\mathfrak{I}(h(\theta))\|_{L^\infty(\TT)}^2}\norm{f_{\refd}(a)}{L^2(\Dref;\CC)}.
\end{equation} 
We also have the estimate
\begin{equation}\label{eq:est-rhs}
\begin{aligned}
\norm{f_{\refd}(a)}{L^2(\Dref;\CC)}
&	= \|J(a)(f\circ F(a))\|_{L^2(\Dref;\CC}) 
\\
	&\le \sqrt{\|J(a)\|_{L^\infty(\Dref;\CC)}}\ \|f\|_{L^2({D_\kappa(a)})}
\\
	&\leq \|a(\theta)\|_{L^\infty(\TT;\CC)}\ \|f\|_{L^2(\RR^2)}.
\end{aligned}
\end{equation}
Thus, we finally arrive at
\begin{equation}\label{hat u <}
  \norm{\hat{u}(a)}{V_\CC} \le \frac{2+\|\mathfrak{R}(h(\theta))\|_{L^\infty(\TT)}^2}
  {1-\|\mathfrak{I}(h(\theta))\|_{L^\infty(\TT)}^2}
  \|a(\theta)\|_{L^\infty(\TT;\CC}) \norm{f}{L^2(\RR^2)}.
\end{equation} 

Let $\Aa\subset W^1_\infty(\TT;\CC)$ be the set of all $a\in W^1_\infty(\TT;\CC)$
which are lower bounded away from zero
and such that 
\begin{equation}\label{BL-Im}
	|\mathfrak{R}(h(\theta))| < \infty, \ \ \ |\mathfrak{I}(h(\theta))| < 1,\ \ \
	h(\theta):= \frac{a'(\theta)}{a(\theta)},\ \ \theta \in \TT.
\end{equation}
From \cite[Theorem 2.1 and Corollary 2.4]{CD} we 
obtain, using \eqref{eq:fAn}, the following.

\begin{lemma} \label{lemma:hol-hat{u}(a)}
For any $0\leq \kappa<1$, 
the pullback solution $\hat{u}(a)\in V_\CC$ exists for every $a \in \Aa$. 
Moreover, the map
\[
	a\mapsto \hat{u}(a)
\]
is holomorphic over a neighborhood of $\mathcal{A}$.
\end{lemma}

\section{Random variations of the domain}
\label{sct:randomPDE}
\subsection{Log-Gaussian shape parametrization}
\label{sct:lognormal}
In the following, we consider the Poisson problem 
\eqref{eq:PoissionD_a} with $D_\kappa(a)$, $0\leq \kappa < 1$ as defined in \eqref{D_a} 
and 
$a$ depending on $\by=(y_k)_{k\in \NN}$ in the log-Gaussian form
\begin{equation}\label{a=exp}
	a(\by)(\theta) = \exp\Bigg(\sum_{k \in \NN}
	y_k  \psi_k(\theta)  \Bigg), \quad\theta \in \TT, \quad
	\by = \brac{y_k}_{k \in \NN} \in \RR^\NN.
\end{equation}
Here $y_k$ are i.i.d.~standard Gaussian random variables 
on $\RR$, and $\psi_k \in W^1_\infty(\TT)$. 
With \eqref{a=exp},
we rewrite \eqref{eq:PoissionD_a} in the parametric form
\[
	-\Delta u(\by) = f\ \text{in $D_\kappa(\by)$},\quad
	u(\by) = 0 \ \text{on $\partial D_\kappa(\by)$},
\]
where 
\[
      D_\kappa(\by):=  {\rm int} 
             \brab{\bx=(r \cos \theta, r \sin \theta): \ 0 \le \kappa a(\by)(\theta) \leq r < a(\by)(\theta) }.
\]
I.e., for any choice of $\by$, $D_\kappa(\by)$ is obtained by a 
    homothetic radial scaling $r\mapsto a(\theta)r$ of $D_{\refd,\kappa}$,
    which is either an annulus when $0<\kappa < 1$ or the unit circle for $\kappa=0$.
    In the latter case, $D_\kappa(\by)$ is contractible and star-shaped with respect to the origin 
    for any choice of $\by$.
Randomness of $D_\kappa(\by)$ will be modelled by 
considering $\by\in \RR^\NN$ endowed with a Gaussian measure $\bgamma$ (GM for short). 

Let $\gamma(y)$ be the standard Gaussian probability measure on $\RR$ 
with the  density 
\begin{equation} \label{g}
	\rho(y):=\frac 1 {\sqrt{2\pi}} e^{-y^2/2}.
\end{equation}
We recall (e.g. \cite{Bogachev}) 
the concept of
standard Gaussian probability measure $\bgamma(\by)$ on $\RR^{\NN}$ 
as countable tensor product of the Gaussian measures $\gamma(y_j)$:
\begin{equation} \label{eq:GM}
	\bgamma(\by) 
	:= \ 
	\bigotimes_{j \in \NN} \gamma(y_j) , \quad \by = (y_j)_{j \in \NN} \in \RR^{\NN}.
\end{equation}
If $X$ is a separable Banach space, 
the standard Gaussian probability measure $\bgamma$ in \eqref{eq:GM} on $\RR^{\NN}$ 
induces  the Bochner space $L^2(\RR^{\NN},X;\bgamma)$ of  
strongly $\bgamma$-measurable mappings $v$ from $\RR^{\NN}$ to $X$, 
equipped with the norm
\begin{equation} \label{L_2(RRi,X,gamma)}
	\|v\|_{L^2(\RR^{\NN},X;\bgamma)}
	:= \
	\left(\int_{\RR^{\NN}} \|v(\cdot,\by)\|_X^2 \, \rd \bgamma(\by) \right)^{1/2}.
\end{equation}
Elements $v\in L^2(\RR^{\NN},X;\bgamma)$ are referred to as 
\emph{
Gaussian random fields (GRFs) taking values in $X$.
}
\begin{example} \label{example}
A typical instance of \eqref{a=exp} that we consider also 
in our numerical experiments, is that in \eqref{a=exp} 
$b(\by) = \log(a(\by))$ is a stationary Gaussian random field on $\TT$.
Then, for fixed $\by$, $\theta \mapsto b(\by)(\theta)$ 
is a Fourier series, i.e.~$a(\by)$ is of the form
\begin{equation}\label{a=trigonometric}
a(\by)(\theta) = \exp\big(b(\by)(\theta)\big),
\end{equation}
where $\by = (y_k)_{k\in \ZZ}$ is a 
sequence of i.i.d.~standard Gaussian normal variables
on a probability space $(\Omega,\Sigma,\mathbb{P})$
and $b(\by)(\theta)$ is a GRF taking values in $C(\TT)$.
We assume the GRF $b(\by)(\theta)$ to be given in the form of a \KL{ }expansion
$$
b(\by)(\theta) = \sum_{k\in \ZZ} y_k(\omega) \sigma_k b_k(\theta) , \; \theta \in \TT\;,
$$
where, due to stationarity,
$$
b_0 = \frac{1}{2\pi} \quad \text{and} \quad
b_k(\theta) \simeq \cos(k\theta), \; b_{-k} \simeq \sin(k\theta)
\ \text{for}\ k\in \NN,
$$
are scaled to be orthonormal with respect to the $L^2(\TT)$ innerproduct 
$(\cdot,\cdot)_{\TT}$ given by 
$(\varphi,\psi)_{\TT} = \int_0^{2\pi} \varphi(\theta) \psi(\theta) \d\theta$, 
i.e.~$(b_k,b_j)_{\TT} = \delta_{kj}$ for $j,k\in \NN$.
Comparing with \eqref{a=exp}, we set 
$\psi_k(\theta) = \sigma_k b_k(\theta)$, 
so that 
$$
b(\by)(\theta) 
= 
y_0 \psi_0 + \sum_{0\ne k\in \ZZ} y_k  \psi_k(\theta) 
\;.
$$
The sequence $\brac{\sigma_k^2}_{k \in \ZZ} \subset [0,\infty)$ 
corresponds to the eigenvalues of the covariance operator of the GRF $b(\by)(\theta)$.
It is a compact and self-adjoint operator on $L^2(\TT)$.
We assume the eigenvalue sequence $\{ \sigma_{\pm k}^2 \}_{k\in \ZZ}$ 
to be enumerated in decreasing order according to 
$\sigma_0 \geq \sigma_{\pm 1} \geq\cdots\geq \sigma_{\pm k} \geq\cdots\ge 0$ 
and accumulating only at zero.
The connection to random Fourier series on $\TT$ is via
\begin{equation}\label{eq:bySer}
b(\by)(\theta) = y_0 {\lambda_0} + \sum_{k=1}^\infty y_k {\lambda_k} \cos (k\theta) + y_{-k} {\lambda_{-k}}\sin (k\theta),
\end{equation}
so that $\lambda_k = \mathcal{O}(\sigma_{k})$.
\begin{remark}\label{rmk:ParDer}
The derivative with respect to $\theta$ of the expansion \eqref{a=trigonometric} is given by
$a'(\by)(\theta) = (a(\by)b'(\by))(\theta)$,
so that in \eqref{h(theta)} we obtain for \eqref{a=trigonometric} that $h(\theta) = b'(\theta)$.
\end{remark}
\begin{figure}[hbt]
\begin{center}
\includegraphics[width=0.25\textwidth]{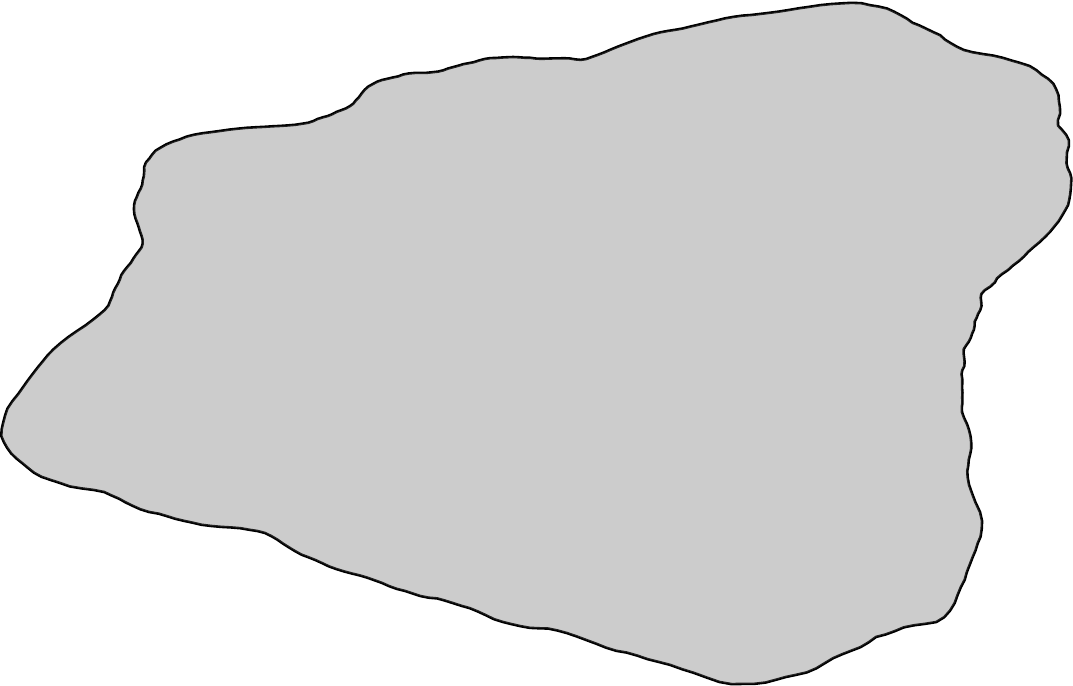}\qquad
\includegraphics[width=0.25\textwidth]{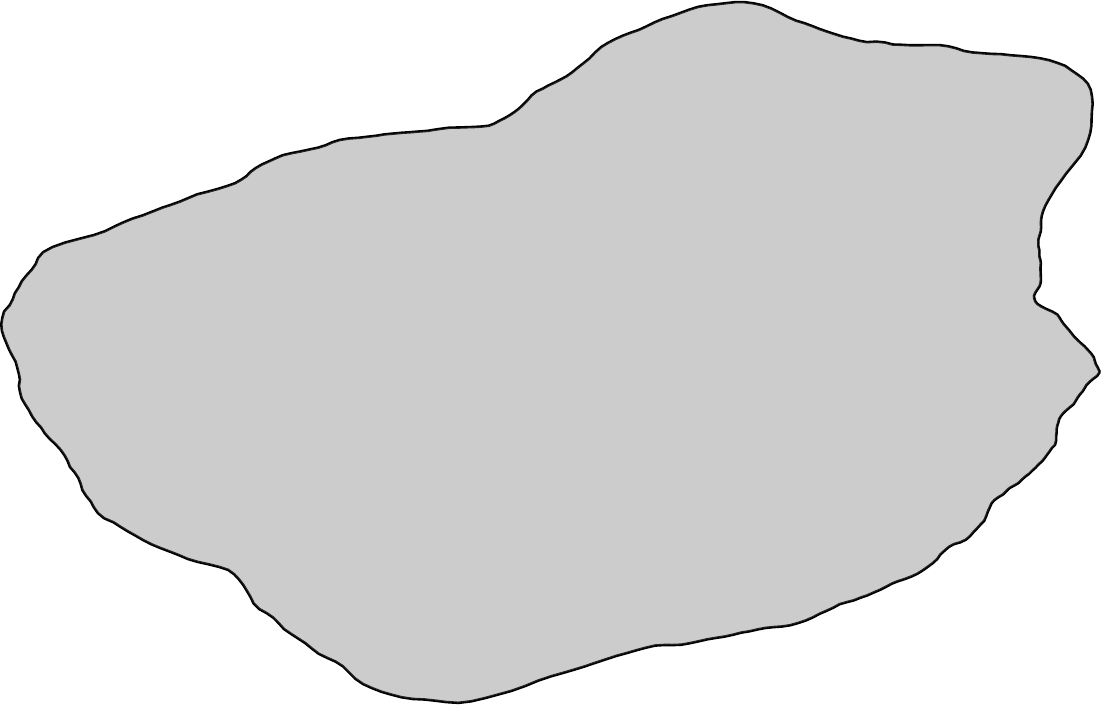}\qquad
\includegraphics[width=0.25\textwidth]{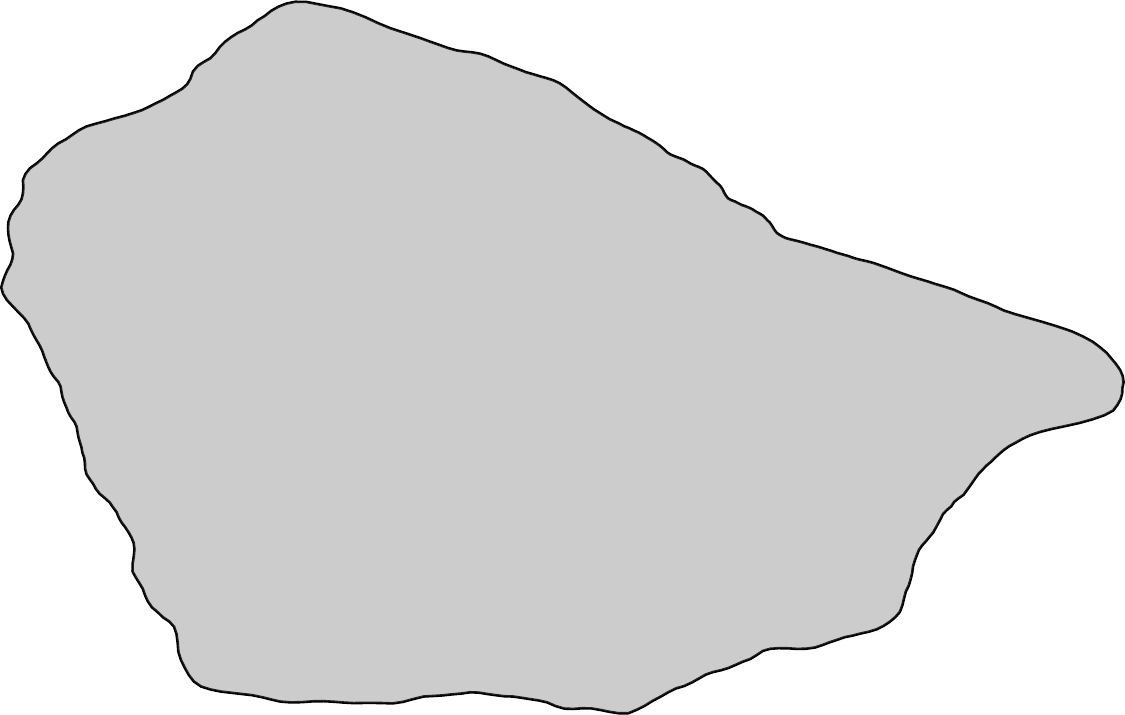}\\[2ex]
\begin{caption}{\label{fig:domain2}
Three realizations of the random domain with log-Gaussian random boundary 
in case of $\kappa = 0$ and for the sequence 
$\lambda_k = (|k|+1)^{-2}$ for all $k\in\mathbb{Z}$.}
\end{caption}
\end{center}
\end{figure}

\begin{figure}[hbt]
\begin{center}
\includegraphics[width=0.25\textwidth]{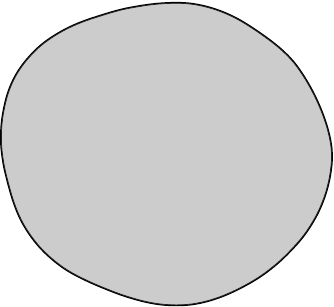}\qquad
\includegraphics[width=0.25\textwidth]{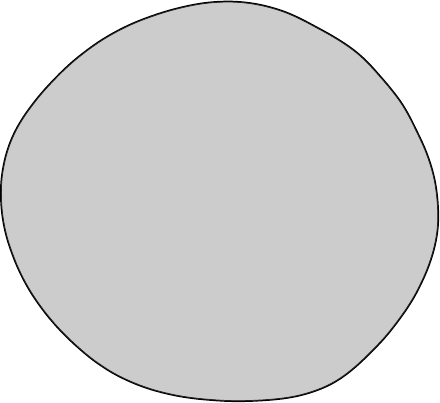}\qquad
\includegraphics[width=0.25\textwidth]{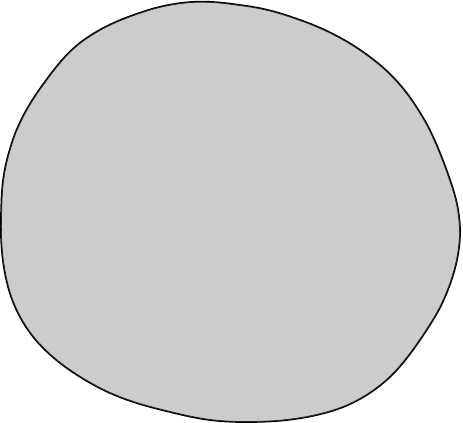}\\[2ex]
\begin{caption}{\label{fig:domain3}
Three realizations of the random domain with log-Gaussian random boundary 
in case of  $\kappa = 0$ and for the sequence 
$\lambda_k = (|k|+1)^{-3}$ for all $k\in\mathbb{Z}$.}
\end{caption}
\end{center}
\end{figure}

A visualization of three random
samples in case  $\kappa = 0$ and the sequence $\lambda_k = (|k|+1)^{-2}$ 
is shown in Figure~\ref{fig:domain2}, 
while a visualization of three random samples in case of $\lambda_k = (|k|+1)^{-3}$ 
is shown in Figure~\ref{fig:domain3}. 
\end{example}

Given $a=a(\by)$, 
we denote by $\hat{u}(\by):=\hat{u}\big(a(\by)\big)$ the 
weak pullback solution to variational formulation \eqref{eq:varformD}.
It satisfies the equation 
\begin{equation}\label{eq:varformD(by)}
\int_{D_{\refd,\kappa}} M(\by) \nabla \hat{u}(\by) \cdot\nabla v  \ = \ \langle f_{\refd}(\by), v\rangle, 
\quad v \in V,
\end{equation} 
where
\begin{equation}\label{M,f}
M(\by):= M\big(a(\by)\big) \ \ \text{and} \ \  f_{\refd}(\by):= f_{\refd}\big(a(\by)\big)
\end{equation}
are the diffusion matrix 
and  the right-hand 
side. Note that equation \eqref{eq:varformD(by)} is the 
variational formulation of the equation
\[
	-\div\big(M(\by)\nabla \hat{u}(\by)\big) 
	= f_{\refd}(\by)\ \text{in $\Dref$},\quad
	\hat{u}(\by) = 0 \ \text{on $\partial D_{\refd,\kappa}$}.
\]

A complex extension of $\hat{u}(\by)$ 
is $\bz\mapsto \hat{u}(\bz)\in V_\CC$ 
which satisfies the equation 
\begin{equation}\label{eq:varformD(bz)}
\int_{D_{\refd,\kappa}} M(\bz) \nabla \hat{u}(\bz) \cdot \nabla \bar{v}\ 
       = \ \langle f_{\refd}(\by),\bar{v}\rangle, 
\quad v \in   V_\CC ,
\end{equation} 
which is the variational formulation of the equation
\begin{equation}\label{eq:problemD(bz)}
	-\div\big(M(\bz)\nabla \hat{u}(\bz)\big) 
		= f_{\refd}(\bz)\ \text{in $\Dref$},\quad
			\hat{u}(\bz) = 0 \ \text{on $\partial D_{\refd,\kappa}$},
\end{equation}
where 
\begin{equation*}
M(\bz):= M\big(a(\bz)\big), \quad f_{\refd}(\bz):= f_{\refd}\big(a(\bz)\big),
\end{equation*}
 and
\[
	a(\bz)(\theta) = \exp\Bigg(\sum_{k \in \NN}
		z_k  \psi_k(\theta)  \Bigg), \quad\theta \in \TT, 
	\quad \bz = \brac{z_k}_{k \in \NN} \in \CC^\NN.
\]

We set
\begin{equation}\label{eq:defb}
	b(\bz)(\theta)  := \log (a(\bz)) = \sum_{k \in \NN} z_k  \psi_k(\theta),  
\end{equation}
where $\log(\cdot)$ denotes the principal branch of the logarithm of a complex argument.
Then we have
\[
h\big(a({\bz})\big)(\theta)= b'({\bz})=\sum_{k\in \NN} z_k \psi_k'(\theta), \ \ \theta \in \TT,
\]
and by \eqref{hat u <}
\begin{equation}\label{norm{hat{u}(by)}{V}}
\begin{aligned}
  \norm{\hat{u}({\bz})}{{V_\CC}} &\le \frac{2+\|\mathfrak{R}(b'({\bz}))\|_{L^\infty(\TT)}^2}
  {1-\|\mathfrak{I}(b'({\bz}))\|_{L^\infty(\TT)}^2}{\|a(\bz)\|_{L^\infty(\TT;\CC)}}
  \norm{f}{L^2(\RR^2)}
\\
&\le \frac{2+\|\mathfrak{R}(b'({\bz}))\|_{L^\infty(\TT)}^2}
{1-\|\mathfrak{I}(b'({\bz}))\|_{L^\infty(\TT)}^2}
\exp\big(\|\mathfrak{R}(b(\bz))\|_{L^\infty(\TT)}\big)\norm{f}{L^2(\RR^2)},	
\end{aligned}
\end{equation} 
provided $\|\mathfrak{I}(b'({\bz}))\|_{L^\infty(\TT)} < 1$.
\subsection{Holomorphy of the parametric pullback solution}
\label{sct:holomorphy}
In \eqref{a=exp}, denote 
\begin{equation}\label{eq:U0}
U_0= \big\{\by \in \RR^\NN:  a(\by) \in W^1_\infty(\TT) \big\} 
   = \big\{\by \in \RR^\NN:  b(\by) \in W^1_\infty(\TT) \big\}.
\end{equation}
In view of \cite[Theorem 2.2]{BCDM17}, we have the following result.

\begin{lemma}\label{lem:full-measure}
Assume that there exists a sequence $(\tau_k)_{k\in \NN}$ 
such that 
$\exp(-\tau_k^2)_{k\in \NN}\in \ell_1(\NN)$ and the series 
\[
\theta \mapsto \sum_{k\in \NN} \tau_k |\psi_k (\theta) |\ \ \text{and} \ \ \	
\theta \mapsto \sum_{k\in \NN} \tau_k |\psi_k'(\theta)|
\]
converge in $L^\infty(\TT)$. 
Then $\bgamma(U_0)=1$.
\end{lemma}

In the following we always assume that the sequence $(\psi_k)_{k\in \NN}$ 
satisfies the conditions in Lemma \ref{lem:full-measure}. 
Let $\brho=(\rho_j)_{j\in \NN} $ be a 
sequence of non-negative numbers and assume that 
$\mfu \subseteq \supp(\brho)$ is finite. 
Define
\[
	\mathcal{S}_\mfu (\brho) := {\bigtimes_{j\in \mfu}} \mathcal{S}_j(\brho),
\]
where the strip $ \mathcal{S}_j (\brho) $ is given by
\[
	\mathcal{S}_j (\brho):= \{ z_j\in
	\CC\,: |\mathfrak{I}(z_j)| < \rho_j\}.
\]
For $\by\in \RR^\NN$, put
\[
\mathcal{S}_\mfu (\by,\brho) := \big\{(z_j)_{j\in \NN}: z_j \in
\mathcal{S}_j(\brho)\ \text{if}\ j\in \mfu\ \text{and}\ z_j=y_j \
\text{if}\ j\not \in \mfu \big\}.
\]

\begin{lemma}\label{lemma:hol}
Let the sequence $\brho=(\rho_j)_{j\in \NN}$ satisfy
\[
	\Bigg\|\sum_{k \in \NN}\rho_k  |\psi_k'| \Bigg\|_{L^\infty(\TT)} <  1.
\]
Let $\by_0=(y_{0,j})_{j \in \NN} \in \RR^{\NN}$ be such that $b(\by_0)$
belongs to $W^1_\infty(\TT)$, and let $\mfu\subseteq \supp(\brho)$ be a
finite set. Assume \eqref{eq:fAn}.

Then the solution $\hat{u}(\by) $ of the parametric variational formulation 
\eqref{eq:varformD(by)} is holomorphic on $ \mathcal{S}_\mfu (\brho) $ 
as a function of the parameters $\bz_\mfu=(z_j)_{j \in \NN} \in 
\mathcal{S}_\mfu (\by_0,\brho)$ taking values in $V$ with 
$z_j = y_{0,j}$ for $j\not \in \mfu$ held fixed.
\end{lemma}

\begin{proof}
Let $N\in \NN$.  We denote
\[
	\mathcal{S}_{\mfu,N} (\brho) 
	:= 
	\big\{ (y_j+\im \xi_j)_{j\in \mfu}\in \mathcal{S}_\mfu (\brho): |y_j-y_{0,j}|<N\big\}\,.
\]
For $\bz_{\mfu} = (y_j+\im \xi_j)_{j\in \ZZ}\in \mathcal{S}_{\mfu}
(\by_0,\brho)$ with $(y_j+\im \xi_j)_{j\in \mfu}\in \mathcal{S}_{\mfu,N} (\brho)$,
we have
\[
	|\mathfrak{I}(b'(\bz_{\mfu})(\theta))| 
	\le 	
	\Bigg\|\sum_{k \in \NN}	\rho_k  |\psi_k'(\theta)| \Bigg\|_{L^\infty(\TT)} 
	<  1,
\]
and
\begin{align*}
	|\mathfrak{R}(b'(\bz_{\mfu})(\theta))|&=\bigg| \sum_{k\in \NN} y_k \psi_k'(\theta)\bigg|
	\\
	&\leq \bigg|\sum_{k\in \NN}y_{0,k}\psi_k'(\theta)\bigg|
	+ \sum_{k\in \mfu}|y_{0,k}-y_k|\cdot |\psi_k'(\theta)|
	\\
	&\leq \|b'(\by_0)\|_{L^\infty(\TT)}+N\sum_{k\in\mfu}|\psi_k'(\theta)| <\infty.
\end{align*}
Similarly we have
\[
\|\mathfrak{R}(b(\bz_\mfu))\|_{L^\infty(\TT)}<\infty.
\]
Hence, by  Lemma~\ref{lemma:hol-hat{u}(a)}, \eqref{eq:fAn} and the analyticity 
of exponential functions, we conclude that the map $\bz_\mfu\mapsto
\hat{u}(\bz_{\mfu})$ is holomorphic on the set $\mathcal{S}_{\mfu,N} (\brho)$. 
Since $N$ is arbitrary, 
we finally deduce that the map $\bz_\mfu\mapsto \hat{u}(\bz_{\mfu})$ 
is holomorphic on $\mathcal{S}_{\mfu} (\brho)$.
\end{proof}
\subsection{Derivative estimates}
\label{sct:decay}
We shall next study the behaviour of the derivatives 
$\partial^{\bnu}\hat{u}(\by)$ of the pullback solution. 
Recall that $\FF$ is the set of all sequences of non-negative 
integers $\bnu=(\nu_j)_{j \in \NN}$ such that their support 
$\supp (\bnu):= \{j \in \NN: \nu_j >0\}$ is a finite set.
The analytic continuation of the parametric solutions
$\{\hat{u}(\by): \by\in U_0 \}$ to $\mathcal{S}_{\mfu} (\brho)$ leads 
to a result on parametric $V$-regularity.
\begin{theorem}\label{thm:derivative-bound}
Assume that there exist  
a non-negative sequence $\brho = (\rho_j)_{j \in \NN}$ 
such that there are constants $B_\brho, C_\brho$ such 
that in \eqref{a=exp} holds
\begin{equation}\label{kappa1}
	\Bigg\|\sum_{k \in \NN}\rho_k  |\psi_k'| \Bigg\|_{L^\infty(\TT)}
	\leq B_\brho < 1
\end{equation}
and
\begin{equation}\label{kappa_0}
	\Bigg\|\sum_{k \in \NN}\rho_k  |\psi_k| \Bigg\|_{L^\infty(\TT)}
	\leq C_\brho.
\end{equation}
Let $\by \in \RR^\NN$ with $b(\by)\in W^1_\infty(\TT)$ 
and 
$\bnu\in \FF$ such that $\supp(\bnu)\subseteq \supp(\brho)$.  

Then
\[
	\|\partial^{\bnu}\hat{u}(\by)\|_V 
	\leq \exp(C_\brho)\frac{\bnu!}{\brho^\bnu}
	\frac{ \big(\|b'(\by )\|_{L^\infty(\TT)} +2\big)^2}{1 - B_\brho} 
	\exp\big(\| b(\by)\|_{L^\infty(\TT)}\big)\|f\|_{L^2(\mathbb{R}^2)}.
\]
\end{theorem}

\begin{proof}
Let $\bnu\in \FF$ be such that
$\supp(\bnu)\subseteq \supp(\brho)$.  
Denote $\mfu=\supp(\bnu)$.
Fixing the variable $y_j$ when $j\not \in \mfu$, 
the map
$\mathcal{S}_\mfu (\by,\brho) \ni \bz_{\mfu}\mapsto \hat{u}(\bz_\mfu)$ 
is holomorphic on the domain 
$\mathcal{S}_\mfu(\by,\brho)$ 
by Lemma~\ref{lemma:hol}.
Applying Cauchy's integral formula gives
\[
	\partial^{\bnu}\hat{u}(\by) =
		\frac{\bnu!}{(2\pi i)^{|\mfu|}}
		\int_{\mathcal{C}_{\by,\mfu}(\brho)} 
		\frac{\hat{u}(\bz_\mfu) }{\prod_{j\in \mfu}  (z_j-y_j)^{\nu_j+1}}\prod_{j\in \mfu}\rd z_j,
\]
where integration is over the cylinders
\[
	\mathcal{C}_{\by, \mfu}(\brho) :=
	\bigtimes_{j\in \mfu} \mathcal{C}_{\by,j}( \brho)\,,\qquad
	\mathcal{C}_{\by,j} ( \brho) := \big\{ z_j \in \CC:
	|z_j-y_j|=\rho_j\big\}\,.
\]
This leads to
\begin{equation} \label{u-y}
	\|\partial^{\bnu}\hat{u}(\by)\|_{V}\leq \frac{\bnu!}{\brho^\bnu}
		\sup_{\bz_\mfu\in \mathcal{C}_\mfu(\by,\brho)}
		\|\hat{u}(\bz_\mfu)\|_{V_\CC}
\end{equation} 
with
\[
	\mathcal{C}_\mfu(\by,\brho)
		=\big\{(z_j)_{j\in \NN} \in \mathcal{S}_\mfu (\by,\brho): \
		(z_j)_{j\in \mfu}\in \mathcal{C}_{\by,\mfu}(\brho) \big\}.
\]
Notice that for $\bz_\mfu=(z_j)_{j\in \NN} \in \mathcal{C}_\mfu(\by,\brho)$ 
we can write $z_j = y_j + \eta_j + \im\xi_j \in \Cc_{\by,j}(\brho)$ with
$|\eta_j | \le \rho_j$, $|\xi_j| \le \rho_j$ if $j\in \mfu$ and $\eta_j = \xi_j=0$ 
if $j\not \in \mfu$. 
We deduce from \eqref{norm{hat{u}(by)}{V}} that
\begin{align*} 
	\|\hat{u}(\bz_\mfu)\|_{V_\CC} 
	& \leq 
	\frac{2+\|\mathfrak{R}\brac{b'(\bz_\mfu)}\|_{L^\infty(\TT)}^2}
	{1-\|\mathfrak{I}\brac{b'(\bz_\mfu)}\|_{L^\infty(\TT)}^2} 
	\exp\big(\|\mathfrak{R}(b(\bz_\mfu))\|_{L^\infty(\TT)}\big)\|f\|_{L^2(\mathbb{R}^2)}
	\\
	& \leq 
	\frac{2+\|b'(\by + \beeta)\|_{L^\infty(\TT)}^2}{1 - B_\brho}
	\exp\big(\| b(\by+\beeta)\|_{L^\infty(\TT)}\big)\|f\|_{L^2(\mathbb{R}^2)}
	\\
	& \leq \exp(C_\brho)
	\frac{2+\big(\|b'(\by )\|_{L^\infty(\TT)} +1\big)^2}{1 - B_\brho}
	\exp\big(\| b(\by)\|_{L^\infty(\TT)}\big)\|f\|_{L^2(\mathbb{R}^2)}	
	\\
 	& \leq \exp(C_\brho)
	\frac{ \big(\|b'(\by )\|_{L^\infty(\TT)} +2\big)^2}{1 - B_\brho}
	\exp\big(\| b(\by)\|_{L^\infty(\TT)}\big)\|f\|_{L^2(\mathbb{R}^2)}.
\end{align*}
By inserting this estimate into \eqref{u-y}, 
we obtain the desired result.
\end{proof}

This theorem immediately implies the following
corollary which concerns the special series expansion 
\eqref{a=trigonometric} which is based on a Fourier series. 
\begin{corollary}\label{corollary:estV}
Consider $a(\by)(\theta)$ as in \eqref{a=trigonometric}, 
with Fourier series $b(\by) = \log(a(\by)(\theta)$.
Assume that there exist  
a non-negative sequence $\brho = (\rho_j)_{j \in \ZZ}$  
and a number $B_\brho > 0$ satisfying  
\begin{equation}\label{kappa2}
	\Bigg\|\sum_{k=1}^\infty 
	\rho_k k |\lambda_k\sin (k\theta)| 
	+ \rho_{-k} k |\lambda_{-k}\cos (k\theta)| \Bigg\|_{L^\infty(\TT)} 
	\leq B_\brho <  1
\end{equation} 
and
\[
C_{\brho} 
:= 
\Bigg\|\rho_0 \lambda_0 + \sum_{k=1}^\infty 
\rho_k |\lambda_k\cos (k\theta)| 
+ \rho_{-k} |\lambda_{-k}\sin (k\theta)| \Bigg\|_{L^\infty(\TT)}
<\infty .
\]

Let $\bnu\in \NN_0^\ZZ$ be such that $|\bnu|_1<\infty$ and $\supp(\bnu)\subseteq \supp(\brho)$. 
Then for every $\by \in \RR^\ZZ$ with $b(\by)\in W^1_\infty(\TT)$ holds
\begin{equation}\label{eq:uybd}
	\|\partial^{\bnu}\hat{u}(\by)\|_V 
	 \leq \exp(C_\brho)\frac{\bnu!}{\brho^\bnu}
	\frac{ \big(\|b'(\by )\|_{L^\infty(\TT)} +2\big)^2}{1 - B_\brho} 
	\exp\big(\| b(\by)\|_{L^\infty(\TT)}\big)\|f\|_{L^2(\mathbb{R}^2)}
\end{equation}
for the parametric, weak solution $\hat{u}(\by) \in V$ 
to \eqref{eq:varformD(by)} with $a(\by)$ 
as in \eqref{a=trigonometric}.

\end{corollary}
\section{Interpolation and Quadrature}
\label{sct:numerixSpG}
We address the numerical solution of the parametric PDE \eqref{eq:varformD(by)}.
In Section~\ref{sct:interpolation}, we
address sparse-grid interpolation with respect to the parameter, 
and 
sparse Gauss-Hermite Smolyak quadrature with respect to the parameter.
In Section~\ref{sct:QMC}, we address 
quasi-Monte Carlo quadrature over the parameter domain.
\subsection{Sparse grid interpolation and quadrature }
\label{sct:interpolation}
We apply the analyticity results in Lemma~\ref{thm:derivative-bound} 
to obtain a weighted $\ell_2$-summability of the Hermite GPC expansion coefficients 
of $u(\by)$ which gives necessary conditions for constructing 
sparse-grid interpolations and establishing a convergence rate 
that is free from the CoD.

Every function $v \in {L^2(\RR^{\NN},X;\bgamma)}$ can be represented  
by the Hermite polynomial chaos (PC) 
expansion
\begin{equation} \label{series}
	v(\by)=\sum_{\bnu \in \FF} v_\bnu\,H_\bnu(\by), \quad v_\bnu \in X,
\end{equation}
where
\begin{equation*}
	H_\bnu(\by)=\bigotimes_{j \in \NN}H_{\nu_j}(y_j),\quad 
        v_\bnu:=\int_{\RR^\NN} v(\by)\,H_\bnu(\by)\, \rd\bgamma (\by), \quad \bnu \in \FF,
	\label{hermite}
\end{equation*}
with $(H_k)_{k\in \NN_0}$ being the  
Hermite orthonormal polynomials in $L^2(\RR,\gamma)$. 
There holds the Parseval's identity 
\begin{equation} \nonumber
	\|v\|_{L^2(\RR^{\NN},X;\bgamma)}^2
	\ = \ \sum_{\bnu \in \FF} \|v_\bnu\|_X^2.
\end{equation}

For $\theta, \lambda \ge 0$, 
define the set 
$\bp(\theta, \lambda):= \brac{p_\bnu(\theta, \lambda) }_{\bnu \in \FF}$ 
by 
\begin{equation*} 
	p_\bnu(\theta , \lambda) := \prod_{j \in \NN} (1 + \lambda \nu_j)^\theta, \quad \bnu \in \FF.
\end{equation*}
For integer $m \in \NN$, we use the following notation: 
\begin{equation} \nonumber
	\FF_m:= \{\bnu \in \FF: \nu_j \in \NN_{0,m},  \ j \in \NN \},
	\ \ \text{where} \ \
	\NN_{0,m} := \{n \in \NN_0: n = 0, m, m+1, \ldots\}.
\end{equation}
\begin{lemma}\label{lemma:weighted-summability}
Assume that there exist  a non-negative sequence $\brho = (\rho_j)_{j \in \NN}$  
and  
positive numbers $ B_\brho $ and $C_{\brho}$ 
satisfying  \eqref{kappa1} and \eqref{kappa_0}, 
and
$(\rho_j^{-1})_{j \in \NN} \in \ell_p(\NN)$ for some $p \in (0,1)$. 
Set $q:= 2p/(2-p) \in (0,2)$.
Let $m \in \NN$ and $\tau, \lambda \ge 0$ be any fixed numbers.
For  $r\in \NN$ and the sequence 
$\bvarrho=(\varrho_j)_{j \in \NN}$ given by
$$
\varrho_j:=\rho_j^{1-p/2}\frac{1}{\sqrt{r!} \norm{ (\rho_j^{-1})_{j \in \NN}}{\ell_p(\NN)}^{p/2}},
$$ 
we define  the Wiener-Hermite weights $\sigma_\bnu$ by
	\begin{equation} \label{sigma_nu}
		\sigma_\bnu^2 :=
		\sum_{|\bnu'|_\infty\leq r } \binom{\bnu}{\bnu'} \bvarrho^{2\bnu'} 
		= \prod_{j \in \NN}\bigg(\sum_{\ell=0}^{r}\binom{\nu_j}{\ell}\varrho_j^{2\ell}\bigg), \ \ \bnu \in \FF.
	\end{equation} 
	Then, for any fixed $r > 2(\tau + 1)m/q$,
        the Hermite coefficients ${\hat{u}}_\bnu$ of the parametric weak solution ${\hat{u}}(\by)$ of \eqref{eq:varformD(by)}
        admit the weighted $\ell_2$-summability
	\begin{equation*} 
		\sum_{\bnu\in \FF_m} \brac{\sigma_\bnu\|\hat{u}_\bnu\|_{V}}^2 <\infty
		\ \ \ \text{with} \ \ \ 
		 \big(	p_\bnu(\tau, \lambda)\sigma_\bnu^{-1}\big)_{\bnu\in\FF_m} \in \ell_{q/m}(\FF_m).
	\end{equation*}
\end{lemma}

\begin{proof}
	This lemma is derived from Theorem~\ref{thm:derivative-bound} 
        in a manner similar to the proofs of  
        Theorem 3.25 and of Remark 3.16 in \cite{DNSZ} by using \cite[Lemma 5.3]{Zung}.
	\hfill
\end{proof}	
For $m \in \NN_0$, let $Y_m = (y_{m;k})_{k \in \pi_m}$ be the increasing sequence of  
the $m+1$ roots of the Hermite polynomial $H_{m+1}$, 
ordered as
$$
y_{m;-j} < \cdots < y_{m;-1} < y_{m;0} = 0 < y_{m;1} < \cdots < y_{m;j} \quad {\rm if} \  m = 2j,
$$
$$
y_{m;-j} < \cdots < y_{m;-1} < y_{m;1} < \cdots < y_{m;j} \quad {\rm if} \  m = 2j - 1,
$$
where 
$$
\pi_m:= 
\begin{cases}
	\{-j,-j+1,\ldots, -1, 0, 1, \ldots,j-1,j \} \ & \ \text{if} \  m = 2j; \\
	\{-j,-j+1,\ldots,-1, 1,\ldots,j-1,j \} \ & \ \text{if} \  m = 2j-1.
\end{cases}
$$
(in particular, $Y_0 = (y_{0;0})$ with $y_{0;0} = 0$).

For a continuous function $v:\RR\to V$ and for $m \in \NN_0$, 
we define the  
Lagrange interpolation operator $I_m$ in the Hermite nodes 
by
\begin{equation} \label{I_(v)}
	I_m(v):= \ \sum_{k\in \pi_m} v(y_{m;k}) L_{m;k}, \quad 
	L_{m;k}(y) := \prod_{j \in \pi_m \ j\not=k}\frac{y - y_{m;j}}{y_{n;k} - y_{m;j}},
\end{equation}		
(in particular, $I_0(v) = v(y_{0;0})L_{0;0}(y)= v(0)$ and $L_{0;0}(y)=1$). 
Notice that $I_m(v)$ is a function on $\RR$ taking values in $X$ and 
interpolating $v$ at $y_{m;k}$, i.e., $I_m(v)(y_{m;k}) = v(y_{m;k})$.   
We define the univariate operator $\Delta_m$ for $m \in \NN_0$ by
\begin{equation} \nonumber
	\Delta_m
	:= \
	I_m - I_{m-1},
\end{equation} 
with the convention $I_{-1} = 0$,  and the univariate operator $\Delta^*_m$ for even $m \in \NN_0$ by
\begin{equation} \nonumber
	\Delta^*_m
	:= \
	I_m - I_{m-2},
\end{equation} 
with the convention $I_{-2} = 0$. 
For a function $v: \RR^\NN \to X$, 
we introduce the tensor product operator $\Delta_\bnu$, $\bnu \in \FF$, 
by
\begin{equation} \label{Delta_bs(v)}
	\Delta_\bnu(v)
	:= \
	\bigotimes_{j \in \NN} \Delta_{\nu_j}(v),
\end{equation}
where the univariate operator
$\Delta_{\nu_j}$ is successively applied to the univariate function 
$\bigotimes_{i<j} \Delta_{\nu_i}(v)$ by considering it as a 
function of  variable $y_j$ with the other variables held fixed.
We define for $\bnu \in \FF$ in the same manner as $\Delta_\bnu$
\begin{equation} \label{eq:Prdct}
	I_\bnu
	:= \
	\bigotimes_{j \in \NN} I_{\nu_j}, \quad
	L_{\bnu;\bk}
	:= \
	\bigotimes_{j \in \NN} L_{\nu_j;k_j}, \quad
	\pi_\bnu
	:= \
	\prod_{j \in \NN} \pi_{\nu_j}.
\end{equation}

For $\bnu \in \FF$ and $\bk \in \pi_\bnu$, let $E_\bnu \subset \FF$ be the subset 
of all $\be$ such that $e_j$ is either $1$ or $0$ if $\nu_j > 0$, 
and $e_j$ is $0$ if $\nu_j = 0$, 
and let, for $U_0\subset \CC^\NN$ in Section~\ref{sct:holomorphy} 
$\by_{\bnu;\bk}:= (y_{\nu_j;k_j})_{j \in \NN} \in U_0$.
Recall 
$|\bnu|_1 := \sum_{j \in \NN} \nu_j$ for  $\bnu \in \FF$. 
It is easy to check that the interpolation operator $\Delta_\bnu$ 
can be represented in the ``combination'' form
\begin{equation} \label{Delta_bs=}
	\Delta_\bnu(v)				
	\ = \
	\sum_{\be \in E_\bnu} (-1)^{|\be|_1} I_{\bnu - \be} (v)
	\ = \
	\sum_{\be \in E_\bnu} (-1)^{|\be|_1} \sum_{\bk \in \pi_{\bnu - \be}} v(\by_{\bnu - \be;\bk}) L_{\bnu - \be;\bk}.
\end{equation}

Let $0 < q < \infty$ and $(\sigma_\bnu)_{\bnu \in \FF}$ be a set of positive numbers. 
For $\xi > 1$,  define the thresholded multi-index set
\begin{equation*} 
	\Lambda(\xi):= \{\bnu \in \FF: \, \sigma_\bnu^q \le \xi\}.
\end{equation*}
With $\Lambda(\xi)$, 
we introduce the sparse-grid interpolation operator $I_{\Lambda(\xi)}$  by
\begin{equation}\label{eq:SpI} 
	I_{\Lambda(\xi)}
	:= \
	\sum_{\bnu \in \Lambda(\xi)} \Delta_\bnu.
\end{equation}
By \eqref{Delta_bs=}, 
we can represent the operator $I_{\Lambda(\xi)}$
in the form (see \eqref{eq:Prdct})
\begin{equation*} 
I_{\Lambda(\xi)}(v)				
\ = \
\sum_{(\bnu,\be,\bk) \in G(\xi)}  (-1)^{|\be|_1} v(\by_{\bnu - \be;\bk})L_{\bnu - \be;\bk},
\end{equation*}
where the ``sparse grid'' is given by
\begin{equation}\label{eq:SpG} 
G(\Lambda(\xi))				
:= \
\big\{(\bnu,\be,\bk) \in \FF \times \FF \times \FF: \, \bnu \in \Lambda(\xi), \ \be \in E_\bnu, \ \bk \in \pi_{\bnu - \be} \big\}.
\end{equation}
The following theorem gives a bound for the convergence rate of the 
semi-discrete approximation of the parametric solution $u$ by the sparse-grid interpolation $I_{\Lambda(\xi)} u$.
\begin{theorem}\label{thm:interpolation}
Under the assumptions of Lemma~\ref{lemma:weighted-summability}, 
consider the sparse-grid interpolation
operator $I_{\Lambda(\xi)}$ for  $q:= 2p/(2-p)$ 
and the weight sequence $(\sigma_\bnu)_{\bnu\in\FF}$ defined by  \eqref{sigma_nu}.

	Then there exists a constant $C > 0$ such that for each $n > 1$, 
we can determine  a number $\xi_n$ so that for the set of points  
$(\by_{\bnu - \be;\bk})_{(\bnu,\be,\bk) \in G(\Lambda(\xi_n))}$, 
it holds that $|G(\Lambda(\xi_n))| \le n$ and 
	\begin{equation*} 
		\|\hat{u} - I_{\Lambda(\xi_n)}\hat{u}\|_{L^2(\RR^{\NN},V;\bgamma)}\leq Cn^{-(1/p - 1)}.
	\end{equation*}
\end{theorem}
\begin{proof}
By Lemma~\ref {lemma:weighted-summability},
the assumptions of  \cite[Corollary 3.1]{Zung} 
hold for $X=V$ with $0 < q<2$ and $q:= 2p/(2 -p)$ for $p \in (0,1)$ 
given as in the assumptions of Lemma~\ref {lemma:weighted-summability}.
Therefore, by applying \cite[Corollary 3.1]{Zung} for $q$ and 
taking account of $1/q -1/2 =1/p - 1$, we prove the theorem.
\hfill
\end{proof} 
We present corresponding results for sparse-grid (Smolyak) quadrature 
    over $\RR^\NN$ with respect to $\bgamma$.
The symmetry of the Gaussian measure $\bgamma$ with respect to co-ordinate reflection 
    $y_j \mapsto -y_j$ for all $j\in \NN$
    implies vanishing of Hermite-pc terms with at least one odd-degree Hermite
    polynomial.  
    To exploit such cancellations in Hermite-pc expansions due to parity,
we use the index set
$$
\FF_\rev := \{\bnu \in \FF: \nu_j \ {\rm  even}, \ j \in \NN \} \subset \FF_2.
$$ 
The interpolation operators $\Delta^*_\bnu$ for 
$\bnu \in \FF_\rev$, $I^*_\Lambda$ for a finite set $\Lambda \subset \FF_\rev$ 
are defined by replacing $\Delta_{\nu_j}$ with $\Delta^*_{\nu_j}$, $j \in \NN$.

If $v$ is a function defined on $\RR$ taking values in the space $V$,  
the function  $I_m(v)$ in  \eqref{I_(v)}  generates the quadrature formula defined as
\begin{equation} \nonumber
	Q_m(v)
	:= \ \int_{\RR} I_m(v)(y) \, \rd \gamma(y)
	\ = \
	\sum_{k=0}^m\omega_{m;k}\, v(y_{m;k}),
\end{equation}	
where
\begin{equation} \nonumber
	\omega_{m;k}
	:=  \int_{\RR} L_{m;k}(y) \, \rd \gamma(y)
	\ = \
	\frac{1}{(m+1) H_m^2(y_{m;k})}.
\end{equation}
We define the univariate quadrature-increment $\Delta^{{\rm Q}}_m$ for even $m \in \NN_0$ by
\begin{equation} \nonumber
	\Delta^{{\rm Q}}_m
	:= \
	Q_m - Q_{m-2},
\end{equation} 
with the convention $Q_{-2} := 0$. 
For  a function $v$ defined on $\RR^\NN$ taking values in $V$, 
we introduce the operator
\begin{equation} \nonumber
	\Delta^{{\rm Q}}_\bnu
	:= \
	\bigotimes_{j \in \NN} \Delta^{{\rm Q}}_{\nu_j},
\end{equation}
in the same manner as $\Delta_\bnu$. 
For $\xi>1$, 
let the set $\Lambda_\rev^*(\xi)$ be defined by 
\begin{equation} \label{Lambda_rev(xi)2}
	\Lambda_\rev^*(\xi):= 
	\{\bnu \in \FF_\rev: \, \sigma_\bnu^{q/2} \le \xi\}.
\end{equation}
We introduce the sparse-grid quadrature operator $Q_{\Lambda_\rev^*(\xi)}$ 
which is generated by the sparse-grid interpolation operator $I^*_{\Lambda_\rev^*(\xi)}$ as follows
\begin{equation} \label{eq:SpQ}
	Q_{\Lambda_\rev^*(\xi)}
	:= \
	\sum_{\bnu \in \Lambda_\rev^*(\xi)} \Delta^{{\rm Q}}_\bnu (v)
	\ = \
	\int_{{\RR^\NN}} I^*_{\Lambda_\rev^*(\xi)} v (\by)\, \rd \bgamma(\by).
\end{equation} 
\begin{theorem}\label{thm:quadrature}
Under the assumptions of 
Lemma~\ref{lemma:weighted-summability}, 
with the index set $\Lambda_\rev^*(\xi)$ as in \eqref{Lambda_rev(xi)2},
consider the sparse-grid Quadrature operator 
$Q_{\Lambda_\rev^*(\xi)}$ 
defined by 
the weight-sequence $(\sigma_\bnu)_{\bnu\in\FF}$ in \eqref{sigma_nu}. 

Then there exists a constant $C$ such that for each $n \in \NN$, 
there exists a number $\xi_n$  such that  
the number of quadrature points $|G(\Lambda_\rev^*(\xi_n))|$ is at most $n$ 
and 
	\begin{equation} \label{u-Q_Lambdau_pde}
		\left\|\int_{{\RR^\NN}}{\hat{u}}(\by)\, \rd \bgamma(\by ) - Q_{\Lambda_\rev^*(\xi_n)}{\hat{u}}\right\|_V
		\ \le \
		Cn^{-(2/p - 3/2)},
	\end{equation}			
	and if $\phi\in V'$
	\begin{equation} \label{u-Q_Lambdau_phi-pde}
		\left|\int_{\RR^\NN}
		\big\langle \phi, \hat{u} (\by) \big\rangle\, \, \rd \bgamma(\by ) 
		-  \big\langle \phi, Q_{\Lambda_\rev^*(\xi_n)} \hat{u} \big\rangle \right|
		\ \le \
		C\|\phi\|_{V'} n^{-(2/p - 3/2)}.
	\end{equation}
\end{theorem}

\begin{proof}
Note that $\FF_\rev \subset \FF_2$.
By Lemma~\ref {lemma:weighted-summability} 
the assumptions of  \cite[Corollary 3.1]{Zung} hold 
for $X=V$ with $0 < q/2<1$ and $q:= 2p/(2-p)$ for $p \in (0,1)$ 
given as in the assumptions of Lemma~\ref{lemma:weighted-summability}.
Therefore, 
by applying \cite[Corollary 4.1]{Zung} for $q/2$ and taking account $2/q -1/2 =2/p - 3/2$,
we prove the theorem.
\hfill  
\end{proof} 
Observe that Theorems~\ref{thm:interpolation} and \ref{thm:quadrature}  
provide convergence rates free from the  ``curse of dimensionality".
\emph{Arbitrary high convergence rates} are possible, 
for sufficiently small summability exponent $p$, i.e., 
with sufficient sparsity of the Hermite expansion coefficients 
$(\hat{u}_\bnu )_{\bnu\in \FF}$.
\subsection{Quasi-Monte Carlo integration}
\label{sct:QMC}
We shall next show that our regularity results also 
allow for convergence rate bounds of 
quasi-Monte Carlo quadrature to compute the
expectation of the random solution of moments of it. 
Since the underlying random variables are Gaussian 
and hence unbounded, the application of the quasi-Monte 
Carlo method requires some special care, compare \cite{QMC2,
QMC3}. We follow here the approach of \cite{HPS} and apply 
the \emph{Halton sequence} as points of integration.

\begin{definition}
Let \(b_1,\ldots,b_m\) denote the first \(m\) prime numbers. 
The \(m\)-dimensional \emph{Halton sequence} is given by 
\[
\boldsymbol{\xi}^i=[h_{b_1}(i),\ldots,h_{b_m}(i)]^\intercal,\quad i=0,1,2,\ldots,
\]
where \(h_{b_j}(i)\) denotes the \(i\)-th element of the \emph{van der Corput 
sequence} with respect to \(b_j\). That is, if \(i=\cdots c_3c_2c_1\) in radix \(b_j\), then 
\(h_{b_j}(i)=0.c_1c_2c_3\cdots\) in radix \(b_j\). 
\end{definition}

The associated quasi-Monte Carlo quadrature rule 
for a given function $v$ defined on $[0,1]^m$ is then 
defined by
\begin{equation}\label{eq:QMC}
\bQ_N v\isdef\frac 1 N \sum_{i=1}^N v(\boldsymbol{\xi}^i),
\end{equation}
where \(N\) denotes the number of samples and 
\(\boldsymbol{\xi}^i\in[0,1]^m\) denotes the $i$-th Halton point.
The quadrature rule \eqref{eq:QMC} is known to give an
approximation to the integral
\[
\bI v\isdef\int_{(0,1)^m} v(\bz)\d\bz
\]
provided that the integrand $v$ is smooth enough.

To obtain a quasi-Monte Carlo method for the integration domain 
\(\mathbb{R}^m\), we map the quadrature points to \(\mathbb{R}^m\) 
by the inverse distribution function. This is equivalent to the transformation 
of the integrals under consideration to the unit cube. 
To that end, we 
define the cumulative normal distribution
\[
\Phi\colon\mathbb{R}\to(0,1),\quad
y\mapsto\Phi(y)\isdef\int_{-\infty}^y\rho(y')\d y'
\]
where the density $\rho$ is as in \eqref{g} and its inverse
\[
\Phi^{-1}\colon (0,1)\rightarrow\mathbb{R}.
\]
Then, 
for a function \(g\in L^1(\mathbb{R};\gamma)\), 
it is well-known that with $\rho$ as in \eqref{g}
\[
\int_{\mathbb{R}}g(y)d\gamma(y) 
=
\int_{\mathbb{R}}g(y)\rho(y)\d y
=
\int_0^1 g\big(\Phi^{-1}(z)\big)\d z
\]
upon the substitution \(z=\Phi(y)\). 
Especially, we have \(g\circ\Phi^{-1}\in L^1\big((0,1)\big)\).
By defining \({\boldsymbol{\Phi}}(\by)\isdef[\Phi(y_1),\ldots,\Phi(y_m)]^\intercal\), we may extend
the above integral transform to the multivariate case, i.e.~\(g\in L^1_\rho(\mathbb{R}^m)\) and
\[
\int_{\mathbb{R}^m}g(\by)\rho(\by)\d \by 
= \int_{(0,1)^m} g\big(\boldsymbol{\Phi}^{-1}(\bz)\big)\d \bz.
\]

In our application, the integrand $g(\by) = \hat{u}(\bxi,\by)$ is
the solution to \eqref{eq:problemD}--\eqref{f_a} and itself depends 
in addition also on the spatial variable $\bxi\in {D_{\refd,\kappa}}$ 
as we are interested in the expectation $\mathbb{E}[\hat{u}(\bxi)]$. 
Moreover, 
in order to be able to apply the quasi-Monte Carlo quadrature rule, 
we need to truncate the stochastic dimension in accordance with
\begin{equation}\label{eq:dim-trunc}
  \hat{u}(\bxi,\by) = \hat{u}_m(\bxi,y_1,\ldots,y_m) 
  	+ \hat{u}_m^c(\bxi,\by).
\end{equation}
Here, $\hat{u}_m$ denotes the solution of the boundary value 
problem \eqref{eq:problemD}--\eqref{f_a} for the $m$-dimensional
model, i.e., 
the model where the series in \eqref{a=exp} is truncated after $m$ terms. 
This
means that $\hat{u}_m(\bxi,\by)$ is understood via zero padding, i.e.
\[
  \hat{u}_m(\bxi,y_1,\ldots,y_m) = \hat{u}(\bxi,y_1,\ldots,y_m,{\bf 0}).
\]
The function 
$\hat{u}_m^c(\bxi) := \hat{u}(\bxi)-\hat{u}_m(\bxi)$ 
reflects the so-called dimension truncation error.
In view of \eqref{eq:dim-trunc}, we have
with $\mathbb{E}$ denoting expectation with respect to $\bgamma$
\[
  \mathbb{E}[\hat{u}(\bxi)] 
  = \int_{(0,1)^m} \hat{u}_m\big(\bxi,\boldsymbol{\Phi}^{-1}(\bz)\big)\d \bz
  + \mathbb{E}[\hat{u}_m^c(\bxi)],
\]
which we are going to approximate by
\[
  \mathbb{E}[\hat{u}(\bxi)] 
  \approx \bQ_N\big(\hat{u}_m(\bxi,\boldsymbol{\Phi}^{-1}(\cdot)\big)
  + \mathbb{E}[\hat{u}_m^c(\bxi)].
\]
Here, the dimension truncation parameter 
$m$ has to be chosen appropriately such that 
the truncation error 
$\|\mathbb{E}[\hat{u}_m^c(\bxi)]\|_V$ 
is sufficiently small. 
The rate of convergence of this error as $m \to \infty$
depends solely on the 
summability of the coefficients in the series \eqref{a=exp}, 
compare \cite[Thm.~4.1]{GK}. 

As shown in \cite{QMC}, the convergence rate of the 
quasi-Monte Carlo quadrature based on Halton points for 
the determination of the expectation $\mathbb{E}[\hat{u}_m]$ 
depends only mildly on the dimensionality $m$ of the random 
parameter under certain properties of the sequence 
$\brho = (\rho_k)_{k\in\NN}$ from \eqref{u-y}.
\begin{theorem}\label{errNQMC}
The quasi-Monte Carlo quadrature using Halton points for 
approximating the expectation $\EE[\hat{u}]$ of the solution 
\(\hat{u}\) to  \eqref{a=exp} provides a 
convergence rate which depends only linearly on the dimensionality 
\(m\) if the sequence $\brho = (\rho_k)_{k\in\NN}$ from \eqref{kappa1}
satisfies \(\rho_k\gtrsim k^{4+\varepsilon}\) 
More precisely, for each \(\delta>0\), the error 
of the quasi-Monte Carlo quadrature with \(N\) Halton points 
satisfies
\begin{equation}\label{eq:QMCerror}
\big\|\EE[\hat{u}]-{\bf Q}_N\hat{u}\big\|_V
\lesssim (m+1)N^{-1+\delta}\|f\|_{L^2({\RR^2})} + \|\mathbb{E}[\hat{u}_m^c]\|_V.
\end{equation}
The constant hidden in this error estimate depends on 
$\varepsilon$ and $\delta$, but 
neither on the dimension $m\in\NN$ nor on the number of points $N\in\NN$.
\end{theorem}

\begin{proof}
Upon observing that $\bnu!\leq |\bnu|!$ for $\bnu \in \FF$,
estimate \eqref{eq:uybd} 
(or \eqref{u-y} with \eqref{norm{hat{u}(by)}{V}})
implies \cite[Eq.~(14)]{QMC} with 
$(\gamma_k)_{k\in\ZZ} = (\rho_k^{-1})_{k\in\ZZ}$,
or, equivalently, 
\cite[Eq.~(18)]{QMC} with $p=2$ 
(the proof given there for $p>2$ remains valid verbatim then).
As $\gamma_k\lesssim k^{-4-\varepsilon}$ implies 
the error estimate in \eqref{eq:QMCerror} by using 
\cite[Thm.~4.3]{QMC}, we conclude the assertion.
\end{proof}

\section{Finite element discretization}
\label{sct:numerix}
\subsection{Discretization in the reference domain $D_\refd$}
\label{sct:FEM1}
We investigate in this subsection the discretization of a general, 
parameter-dependent, second order, elliptic  boundary value 
problem by the finite element method. 
To this end, let $D_{\refd} \subset \mathbb{R}^2$ 
denote a generic 
bounded smooth reference domain
(such as $D_\refd = D_{\refd,\kappa}$ in \eqref{eq:Dref}, for $0 < \kappa < 1$ 
	but is not necessary for the definition of the finite element discretization).
Throughout this section, the parameters $\by$ will be real-valued.
We recall the set $U_0\subset \RR^\NN$ from \eqref{eq:U0}.

In $D_\refd$, 
we introduce a quasi-uniform, regular, simplicial partition 
$\mathcal{T}_h$ and consider 
continuous, piecewise linear Lagrangian finite element 
spaces spanned by the usual ``hat''-function bases
$\{\varphi_i\}_{i=1}^N$ defined on $\mathcal{T}_h$.
The finite element space is denoted by
\[
V_h = \operatorname{span}\{\varphi_i:i=1,\ldots,N\}
\subset V.
\]
In order to avoid errors by the geometry approximation,
we consider Zenisek's curved triangles in case of a 
non-polygonal boundary, compare \cite{Z}. We however
mention that a piecewise linear approximation of the
boundary is consistent and the subsequent theory can
be extended to this situation straightforwardly by using
standard arguments from \cite{Brenner/Scott} or 
\cite[Thm.~III.1.7]{Braess}.
We also assume for now exact integration in the stiffness matrix and the load vector.
With these assumptions,
given $\by\in U_0\subset \RR^\NN$, 
we define the parametric bilinear form $B(\cdot,\cdot;\by):V\times V\to \RR$
\begin{equation}\label{myproblem1}
	B(\hat{u} , v ;\by) := \int_{D_{\refd}} M(\by)\nabla \hat{u}(\by) \cdot \nabla v
\end{equation}
and the parametric linear form $L(\cdot;\by): V\to \RR$
\begin{equation}\label{myproblem2}
	L(v;\by) := \langle f(\by),v\rangle.
\end{equation}
\begin{remark}\label{rmk:U1}
	In the analysis of first order finite element approximation errors in $D_{\refd,\kappa}$, 
	convergence rates require higher Sobolev regularity of parametric solutions 
        $\hat{u}(\by)$ on $D_{\refd,\kappa}$.
	This, in turn, 
        imposes stronger regularity requirements on the map $F(a)$ and also on 
	the map $\by \mapsto a(\by)$.
	
		For $a$ as in \eqref{a=exp},
		we have $a(\by)\in W^2_\infty(\TT)$ for all $\by\in U_1 \subset \RR^\NN$ 
		where $U_1$ is measurable and of full measure with respect to 
		the tensor-product Gaussian measure $\bgamma$ on $\RR^\NN$ if, 
		in addition to the conditions in Lemma \ref{lem:full-measure}, 
	there also holds
	\begin{equation}\label{eq:U1}
		\Bigg\|\sum_{k \in \NN} \tau_k|\psi_k''| \Bigg\|_{C(\TT)} < \infty.
	\end{equation}
        For $a(\by)\in W^{2,\infty}(\TT)$, 
        the diffusion coefficient $M(\by)$ in \eqref{eq:varformD(by)}
        has entries $M_{ij}(\by)(x)$ which are (also $\bgamma$-a.s.) 
        Lipschitz continuous functions of $x\in \overline{D_{\refd,\kappa}}$,
        for $0 < \kappa < 1$.
        Due to $f\in L^2(D_{\refd,\kappa})$, 
        for $0<\kappa<1$ then $\hat{u}(\by) \in H^2(D_{\refd,\kappa})$. 
        For this regularity and for $a$ as in \eqref{a=exp}, 
        the condition $\kappa > 0$ is essential.
        The case $\kappa = 0$ is discussed in Section~\ref{sct:FEM3} ahead.
\end{remark}
We assume the following that 
the eigenvalues of diffusion matrix $M(\by)(\bxi)$ 
from \eqref{M,f} satisfy for $\by\in U_0$ and $\bxi\in D_\refd$
\begin{equation}\label{eq:lminmax}
0 < \lambda_{\min}(\by)\le\lambda_{\min}\big(M(\by)(\bxi)\big)
\le \lambda_{\max}\big(M(\by)(\bxi)\big) 
\le\lambda_{\max}(\by)<\infty\;.
\end{equation}
Also, the right-hand side $f(\by)\in L^2(D_{\refd})$.
\begin{remark}\label{rmk:minmax}
The preceding assumption holds for \eqref{a=exp}.
From the bounds \eqref{lambda_min}, \eqref{lambda_max} 
we find with  $a = \exp(b)$ as in \eqref{a=exp} that (cf. \eqref{h(theta)})
$h(\by) := a'(\by)/a(\by) = b'(\by)$ and
definition \eqref{eq:U0} of $U_0$ that
\begin{equation}\label{eq:minmaxbd}
\forall \by\in U_0: \quad 
\max\{\lambda_{\min}(\by)^{-1} , \lambda_{\max}(\by) \}
\leq 
2 + \| b'(\by) \|^2_{L^\infty(\TT)}  < \infty 
\;.
\end{equation}
For any $r>0$ exists $C_r>0$ such that
\begin{equation}\label{eq:minmaxbdr}
\forall \by\in U_0: \quad 
\max\{\lambda_{\min}(\by)^{-r} , (\lambda_{\max}(\by))^r \}
\leq 
C_r (1 + \| b'(\by) \|^{2r}_{L^\infty(\TT)})  < \infty 
\;.     
\end{equation}
\end{remark}

For $\by\in U_0$, we define $\hat{u}_h(\by)$ as solution of the parametric Galerkin equations:
\begin{equation}\label{myproblem3}
	\text{Seek $\hat{u}_h(\by)\in V_h$ such that}\
	B\big(\hat{u}_h(\by),v_h;\by\big) = L(v_h,\by)
	\ \text{for all $v_h\in V_h$}.
\end{equation}

\begin{theorem} \label{FeErr}
	For $\by \in U_1$ as defined in Remark~\ref{rmk:U1},
        the finite element solution $\hat{u}_h(\by)$ to 
	\eqref{myproblem1}--\eqref{myproblem3}
	satisfies the error estimate
	\begin{equation}\label{errorFEM1}
		\|\hat{u}(\by)-\hat{u}_h(\by)\|_V
		\lesssim \sqrt{\frac{\lambda_{\max}(\by)}{\lambda_{\min}(\by)}}
		\frac{c(\by)}
		{\lambda_{\min}^2(\by)} h\|f(\by)\|_{L^2(D_{\refd})},
		\quad 0<h\le h_0,
	\end{equation}
	for some $h_0 >0$.
        Here the constant hidden in $\lesssim$ is independent of $\by$, 
        and on
	\begin{equation}\label{M:W1,infty}
		c(\by) := \big\|M(\by)\big\|_{W^1_\infty(D_{\refd})}.
	\end{equation}
	Here, 
	$\big\|M(\by)\big\|_{W^1_\infty(D_{\refd})}$ 
	is a norm of the matrix of $W^1_\infty(D_{\refd})$-norms
	of elements $M_{ij}(\by)$ of $M(\by)$. 
\end{theorem}

\begin{proof}
	In view of the ellipticity and continuity 
	constants $\lambda_{\min}(\by)$ and $\lambda_{\max}(\by)$, 
	respectively, 
        we obtain that for every $\by\in U_0$ there exists a unique 
        solution $\hat{u}_h(\by)\in V_h$ of \eqref{myproblem3}.
        C\'ea's lemma implies for every fixed $\by\in U_0$ 
        for the Galerkin solution 
        $\hat{u}_h(\by)\in V_h$ of \eqref{myproblem1}--\eqref{myproblem3} 
        the error estimate
	\[
	\|\hat{u}(\by)-\hat{u}_h(\by)\|_V
	\lesssim \sqrt{\frac{\lambda_{\max}(\by)}{\lambda_{\min}(\by)}}
	\inf_{v_h\in V_h}\|\hat{u}(\by)-v_h\|_V
	\]
	with a generic constant in $\lesssim$ that is independent of $\by$.
	If $\hat{u}(\by)\in H^2(D_{\refd})$, we thus obtain 
	\begin{equation}\label{eq:intermediate}
		\|\hat{u}(\by)-\hat{u}_h(\by)\|_V
		\lesssim \sqrt{\frac{\lambda_{\max}(\by)}{\lambda_{\min}(\by)}}
		h \|\hat{u}(\by)\|_{H^2(D_{\refd})}, \quad 0<h\le h_0,
	\end{equation}
	for some $h_0 >0$,
	by standard interpolation error estimates, see 
	\cite{Braess,Brenner/Scott} for example. 
	Since the 
	problem under consideration is $H^2(D_{\refd})$ regular
	if the coefficient matrix satisfies $M(\by)\in W^1_\infty(D_{\refd})$,
	we find by tracking carefully the constants in the proofs of 
	\cite[Thms.~8.8 \& 8.12]{Trud} or \cite[Sct.~6.3, Thms.~1 \& 4]{Evans}
	\begin{equation}\label{eq:hatuH2}
	\|\hat{u}(\by)\|_{H^2(D_{\refd})}
	\lesssim \frac{c(\by)}{\lambda_{\min}(\by)}
	\bigg(\|f(\by)\|_{L^2(D_{\refd})} + \|\hat{u}(\by)\|_{H^1(D_{\refd})}\bigg).
	\end{equation}
	Herein, the constant hidden in $\lesssim$ is again independent of $\by$,
        and $c(\by)$ has finite expectation under $\bgamma$.
	By using the standard stability estimate 
	\[
	\|\hat{u}(\by)\|_{V}
	\lesssim \frac{1}{\lambda_{\min}(\by)}\|f(\by)\|_{L^2(D_{\refd})},
	\]
        with constant in $\lesssim$ depending only in $D_{\refd}$
        we finally arrive at
	\[
	\|\hat{u}(\by)\|_{H^2(D_{\refd})}
	\lesssim \frac{c(\by)}{\lambda_{\min}^2(\by)}
	\|f(\by)\|_{L^2(D_{\refd})}.
	\]
	By inserting this estimate into \eqref{eq:intermediate}
	we arrive at \eqref{errorFEM1}.
\end{proof}
\subsection{Fully discrete finite element method}
\label{sct:FEM2}
If the expectation of the $\by$-dependent right-hand side of
the estimate \eqref{errorFEM1} is finite with respect to the 
Gaussian measure, then \eqref{errorFEM1} implies that 
we can solve the problem under consideration on a fixed 
finite element mesh with a sufficiently small mesh size $h$. 
Especially, due to Galerkin orthogonality, the solution by 
piecewise linear Lagrangian finite elements does not 
change if we replace the diffusion matrix $M(\by)$ on 
each finite element triangle $T$ by its mean. 
Nonetheless, in practice, we have to apply numerical quadrature. 

In what follows, we investigate the impact of numerical 
quadrature in form of the midpoint rule. 
This means that we replace the original variational formulation 
\eqref{myproblem1}--\eqref{myproblem3} by
\begin{equation}\label{myproblem3'}
	\text{Seek $\hat{u}_h(\by)\in V_h$ such that}\
	B_h\big(\hat{u}_h(\by),v_h;\by\big) = L_h(v_h,\by)
	\ \text{for all $v_h\in V_h$}
\end{equation}
with the bilinear form $B_h(\cdot,\cdot;\by)$ which is derived 
from the bilinear form $B(\cdot,\cdot;\by)$ by approximating 
all integrals over triangles $\triangle \in \mathcal{T}_h$ 
with the midpoint rule and likewise for the linear form $L_h(\cdot;\by)$.
\begin{theorem}\label{FeErrQ}
		Let $B_h(\cdot,\cdot;\by)$ and $L_h(\cdot;\by)$ be the fully
	discrete versions of $B(\cdot,\cdot;\by)$ and $L(\cdot;\by)$ 
	obtained by application of the midpoint rule. 

Then for $\by\in U_1$ as in Remark~\ref{rmk:U1},
     the parametric FE-solution $\hat{u}_h(\by)\in V_h$ 
     to the fully discrete problem \eqref{myproblem3'} 
     satisfies the error estimate
	\begin{equation}\label{errorFEM2}
		\|\hat{u}(\by)-\hat{u}_h(\by)\|_V \lesssim C(\by) h
		\|f(\by)\|_{W^1_\infty(D_{\refd})}
	\end{equation}
        with the constant hidden in $\lesssim$ independent of $\by$ and
	\[
	C(\by) = \bigg(1+\frac{\lambda_{\max}(\by)}{\lambda_{\min}(\by)}\bigg)
	\Bigg(\sqrt{\frac{\lambda_{\max}(\by)}{\lambda_{\min}(\by)}}
	\frac{c(\by)}{\lambda_{\min}^2(\by)}
	+ c({\by}) + 1\Bigg).
	\]
\end{theorem}

\begin{proof}
	The midpoint rule applied to the restriction of bilinear form 
	$B(v_h,v_h;\by)$ to an arbitrary triangle $\triangle \in \mathcal{T}_h$ 
	of the triangulation $\mathcal{T}_h$ is given by
	\[
	B_h\big|_\triangle(v_h,v_h;\by) = |\triangle| \nabla v_h(\bxi_\triangle)^\intercal 
	M(\by)(\bxi_\triangle)\nabla v_h(\bxi_\triangle),
	\]
	where the size $|\triangle|$ of the triangle $\triangle$ is
	the quadrature weight and the barycenter $\bxi_\triangle \in\triangle$ 
	of $\triangle$ is the only quadrature point. 
	Observe that $\nabla v_h|_\triangle$ is constant in $\triangle$.
	We find
	\begin{equation}\label{eq:CoerT}
		B_h\big|_\triangle(v_h,v_h;\by)	
		\ge |\triangle|\lambda_{\min}(\by)\|\nabla v_h(\bxi_\triangle)\|_2^2
		\\
		\simeq \lambda_{\min}(\by)\|\nabla v_h\|_{L^2(\triangle)}^2
	\end{equation}
	with a constant hidden in $\simeq$ that is independent of the mesh size $h$
	and of the parameter $\by$. 
	Summing \eqref{eq:CoerT} over all $\triangle\in \mathcal{T}_h$,
	the bilinear form $B_h(\cdot,\cdot;\by)$ is uniformly elliptic, i.e.,
	\[
	\forall v_h \in V_h: \quad B_h(v_h,v_h;\by)\gtrsim \lambda_{\min}(\by)\|v_h\|_V^2
	\]
	independent of the mesh size $h$. 
	The parametric matrix of the 
		linear system of equations which corresponds to 
		\eqref{myproblem3'} is,  for fixed $\by\in U_0$
		positive definite and symmetric
                uniformly with respect to $h$, 
		with smallest eigenvalue bounded from below by a constant
		that depends only on the quotient of 
		$\lambda_{\max}(\by)$ and $\lambda_{\min}(\by)$.
        Hence, for every $\by\in U_0$ exists a unique solution $\hat{u}_h\in V_h$ of 
            \eqref{myproblem3'}.
	Strang's first lemma (e.g.~\cite{Braess,Brenner/Scott,Ci}) implies 
	\begin{equation}\label{eq:Strang}
		\begin{aligned}
			&	\|\hat{u}(\by)-\hat{u}_h(\by)\|_V 
                        \lesssim \bigg(1+\frac{\lambda_{\max}(\by)}{\lambda_{\min}(\by)}\bigg)
			\inf_{v_h\in V_h}\bigg\{\|{\hat{u}(\by)}-v_h\|_V
                        \\
			&\quad+ \sup_{w_h\in V_h} \frac{|B(v_h,w_h;\by)-B_h(v_h,w_h;\by)|}{\|w_h\|_V}
			+ \sup_{w_h\in V_h} \frac{|L(w_h;\by)-L_h(w_h;\by)|}{\|w_h\|_V}\bigg\}.
		\end{aligned}
	\end{equation}
        We bound the three terms on the right-hand side. 
        The first term is the usual best-approximation error which is $\mathcal{O}(h)$ due to \eqref{eq:hatuH2} 
        using approximation properties of first order Lagrangian Finite Elements 
        on regular, quasi-uniform triangulations $\mathcal{T}_h$ of $D_\refd$, 
        and the (assumed) regularity $\hat{u}(\by)\in H^2(D_\refd)$.

        For $\big\|M(\by)\big\|_{ W^1_\infty(D_{\refd}) } < \infty$
        the midpoint rule on $\triangle \in \mathcal{T}_h$ 
	has the accuracy\footnote{ 
	We exploit here only the first order derivatives of the 
	integrand instead of the second order derivatives as this is 
	sufficient to get $\mathcal{O}(h)$ accuracy.}
	\[
	\Big|B\big|_\triangle(v_h,w_h;\by)-B_h|_\triangle(v_h,w_h;\by)\Big|
	\lesssim h |\triangle| \big\|M(\by)\big\|_{{W^1_\infty(\triangle)}}
	\|\nabla v_h\|_{{W^1_\infty}(\triangle)}\|\nabla w_h\|_{{W^1_\infty}(\triangle)}.
	\]
        
	Summing over $T\in \mathcal{T}_h$ bounds the second term as
	\begin{equation}\label{eq:B_h}
		\big|B(v_h,w_h;\by)-B_h(v_h,w_h;\by)\big|
		\\
		\lesssim h c({\bf y}) \|v_h\|_{V}\|w_h\|_{V}
	\end{equation}
	with the constant $c(\by)$ defined in\eqref{M:W1,infty}. 
     
        For the third term in the bound \eqref{eq:Strang} 
	the midpoint rule for $\triangle \in \mathcal{T}_h$ yields
	\[
	L_h\big|_\triangle(w_h;\by)
	= |\triangle| \big(f(\by)\big)(\bxi_\triangle)w_h(\bxi_\triangle).
	\]
	Hence, we arrive at the element-wise error bound
	\[
	\Big|L\big|_\triangle(w_h;\by)-L_h\big|_\triangle(w_h;\by)\Big|
	\\
	\lesssim h |\triangle|\|f(\by)\|_{W^1_\infty(\triangle)}
	\|w_h\|_{W^1_\infty(\triangle)}.
	\]
	Summing over $T\in \mathcal{T}_h$, we arrive at
	\begin{equation}\label{eq:L_h}
		\big|L(w_h;\by)-L_h(w_h;\by)\big|\\
		\lesssim h \|f(\by)\|_{W^1_\infty(D_\refd)}\|w_h\|_{V}.
	\end{equation}
	Inserting \eqref{errorFEM1}, \eqref{eq:B_h}, and \eqref{eq:L_h}
	into \eqref{eq:Strang} and employing the bound
	\[
	\|f(\by)\|_{L^2(D_\refd)}
	\lesssim \|f(\by)\|_{W^1_\infty(D_\refd)} 
	\]
	yields \eqref{errorFEM2}.
\end{proof}

Again, if the expectation of the right-hand side of the error 
estimate \eqref{errorFEM2} is finite with respect to the 
Gaussian measure, the proposed fully discrete finite element 
method is sufficient to approximate the parametric solution $\hat{u}(\by)$ 
in expectation.
\subsection{Application to PDEs on log-Gaussian domain}
\label{sct:FEM3}
The aforementioned theory of the fully discrete finite 
element method applies for instance directly to the solution 
of \eqref{eq:varformD_a} with log-Gaussian scaling $a$ as in 
\eqref{a=exp}, i.e.
$a(\by)(\theta) = \exp(b(\by)(\theta))$.

This means that it has been proven that piecewise linear ansatz
functions on a fixed finite element mesh in combination with a 
midpoint rule is sufficient to compute the expected solution at 
optimal rate $\mathcal{O}(h)$ with respect to the energy norm,
in expectation with respect to the Gaussian measure $\bgamma$ 
over the parameter set $U_1$ defined in Remark~\ref{rmk:U1}. 
In particular, 
\emph{numerical integration order} and \emph{\KL{ }truncation}
need not be path-dependent.
Due to \eqref{eq:lminmax} and the bounds \eqref{eq:minmaxbd}, 
for Example~\ref{example}, with small probability 
large values for $\lambda_{\max}(\by)/\lambda_{\min}(\by)$ can arise,
with corresponding large values of the condition number of the 
matrix corresponding to the parametric bilinear form 
$B_h(\cdot,\cdot;\by)$ in \eqref{myproblem3'}.

In our particular situation, 
we use the domain parametrization \eqref{a=exp} in \eqref{D_a}.
The generic reference domain $D_{\refd}$ in Theorems~\ref{FeErr}, \ref{FeErrQ} 
is any of $D_{\refd,\kappa}$ for $0\leq \kappa <1$. 
By Remark~\ref{rmk:minmax} and \eqref{M_a}, 
for $\by\in U_1\subset \RR^\NN$ it holds 
$M_{ij}(\by) \in C(\overline{D_{\refd,\kappa}}))$
for $0<\kappa<1$ and $i,j=1,2$. 
Furthermore, 
for $c(\by)$ in \eqref{M:W1,infty} holds $\mathbb{E}[c(\cdot)] < \infty$.
Also \eqref{M_a} shows that in this case $M_{ij}(\by) \in C(\TT)$ for $\by\in U_1$.

The polar parametrization \eqref{eq:dom_map} of the random geometry map 
which underlies the homothetic transformation can induce singularities
in $M(a)$ at the origin when $\kappa = 0$.
On the one hand, 
we can estimate for $\by\in U_1$ in Remark~\ref{rmk:U1}
\begin{align*}
	\|f(\by)\|_{W_\infty^1(D_{\refd})} 
	&= \|J(\by) (f\circ F)(\by)\|_{W_\infty^1({D_{\refd,\kappa}})} \\
	&\lesssim \|J(\by)\|_{W_\infty^1({D_{\refd,\kappa}})}\|(f\circ F)(\by)\|_{W_\infty^1({D_{\refd,\kappa}})} \\
	&\lesssim \|a(\by) \|_{W_\infty^1(\TT)}^2 \|f\|_{W_\infty^1(\RR^2)}
\end{align*}
which bounds 
$\|f(\by)\|_{W_\infty^1(D_\refd)}$ in \eqref{eq:L_h} deterministically for $\by\in U_1$.

On the other hand, for $\kappa=0$, 
the domain mapping \eqref{eq:dom_map} 
does not amount to a coefficient matrix $M(a)$ 
that is in $W_{\infty}^1(D_{\refd,0})$ -- 
only for $0<\kappa<1$ it is in 
$C^1\big(\overline{D_{\refd,\kappa}}\big)$ provided that $a \in C^2(\TT)$. 
Indeed, 
when differentiating \eqref{M_a} (with respect to $\bxi$), 
one arrives at a matrix that contains the factor $1/r$, 
which implies an $1/r$-singularity at the origin, 
compare \eqref{M_a}. 
Hence, the constant $c(\by)$ in \eqref{M:W1,infty} depends on $1/\kappa$ 
and the estimate \eqref{errorFEM1} holds 
only in $D_{\refd,\kappa}$ for $\kappa > 0$.
As a consequence, 
on $D_{\refd,0}$ a uniform mesh $\mathcal{T}_h$ 
yields only a reduced order of convergence of the finite element method. 
In our numerical experiments,
we refine the finite element mesh towards the origin to overcome this obstruction.
\begin{remark}\label{rmk:OptinE}
For random shapes generated by a log-Gaussian 
parametrization \eqref{a=trigonometric}, 
$\by$-uniform bounds in the
FE error analysis with iterative solvers can not be expected.
However,
the \emph{numerical complexity scales optimally in expectation 
(i.e. on average with respect to $\bgamma$)}, 
i.e., linearly with respect to the number $N$ of degrees of
freedom, when using multigrid solvers \cite{Lukas}.
\end{remark}
\section{Numerical experiment}
\label{sct:experiment}
In our numerical experiment, we consider the 
right-hand side $f({\bx})=\exp(-\|\bx\|_2)$
 and the domain perturbation \eqref{eq:dom_map} with $\kappa=0$ and
\begin{equation}\label{eq:Fourier}
a(\by)(\theta) = \exp\Bigg(y_0\lambda_0 + \sum_{k=1}^\infty 
y_k  \lambda_k \cos (k\theta) 
+ y_{-k} \lambda_{-k}\sin (k\theta)  \Bigg)
\end{equation}
as in Example~\ref{example}. 
The particular series $(\lambda_k)_{k\in \ZZ}$ we prescribe is
\begin{equation}\label{eq:lkEple}
\lambda_0 = \frac{1}{8} \quad\text{and}\quad 
\lambda_k = \frac{1}{(|k|+1)^3}\ \text{for}\ k\in\mathbb{Z}\setminus\{0\}.
\end{equation}
The sum in the Fourier series \eqref{eq:Fourier} is truncated 
after $K=199$ terms, which yields the dominant 399 terms of
the series, i.e.~we used
\[
a^K(\by)(\theta) := \exp\Bigg(b^K(\by)(\theta)\Bigg) 
\]
in our numerical experiments. 
This was sufficiently accurate for our tests, and
choosing larger values of $K$ does not change the results.
Remark that in Theorem~\ref{errNQMC}, $m = \mathcal{O}(K)$.

Accordingly, besides the errors due to FE discretization error and numerical integration
in computing the bilinear form  $B_h(\cdot,\cdot;\by)$ in \eqref{myproblem3'}, 
we incur a further error due to replacing $a$ by $a^K$, resulting in the matrix 
$M^K(\by) := M(a^K(\by))$ with $M(a^K)$ as defined in \eqref{M_a}.

The summability condition is
satisfied with ${B_\brho}\approx 0.876$ if we plug $\rho_k = 1$
for all $k\in\ZZ$ into \eqref{kappa2}. 
However, the decay of 
the series $(\lambda_k)_{k\in\ZZ}$ is slower than required to
satisfy the assumptions of Theorem~\ref{errNQMC} for the 
convergence of the quasi-Monte Carlo method. Indeed, it has 
already been observed in \cite{QMC} that the resulting condition 
for $(\rho_k)_{k\in\ZZ}$ is not sharp.

We introduce a shape regular triangular mesh of the unit disk
$D_{\refd,0}$ by using Zenisek's curved triangles, see \cite{Z}.
The mesh is a-priorily graded towards the origin as seen in 
Figure~\ref{fig:realizations}. On this triangulation, we define 
piecewise linear Lagrangian finite element functions. For 
our experiments, we use a triangulation with about 20000 
triangles which leads to about 10000 finite element basis 
functions. The solver we use is based on a conjugate 
gradient method which is preconditioned with the BPX 
preconditioner, compare \cite{BPX,Brenner/Scott}.
Four realizations of the random domain under 
consideration and the associated solution of the Poisson 
problem are found in Figure~\ref{fig:realizations}, where 
the mesh refinement towards the origin are clearly visible.

\begin{figure}[hbt]
	\begin{center}
		\includegraphics[width=0.45\textwidth,trim={170 40 150 30},clip]{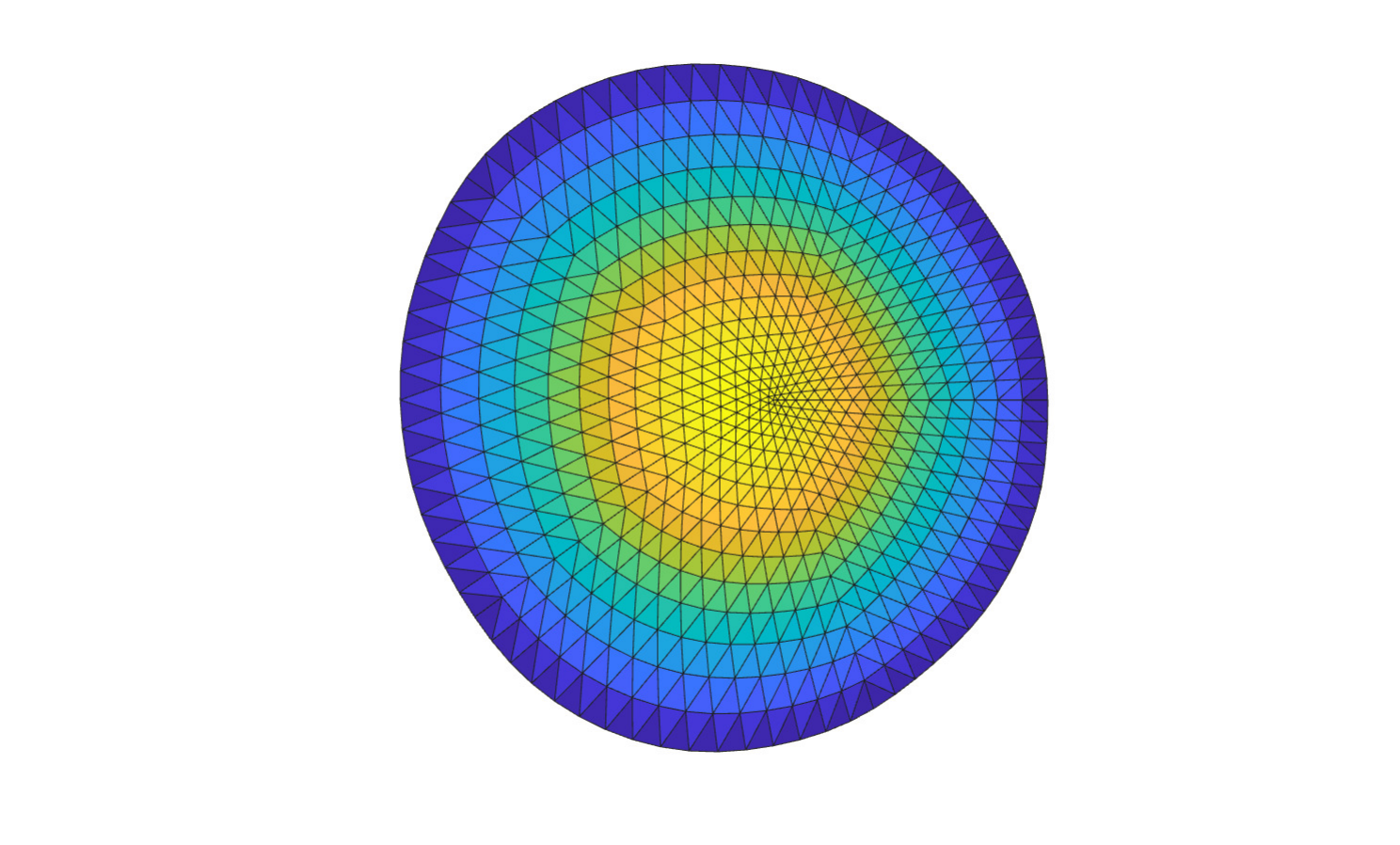}\qquad
		\includegraphics[width=0.45\textwidth,trim={170 40 150 30},clip]{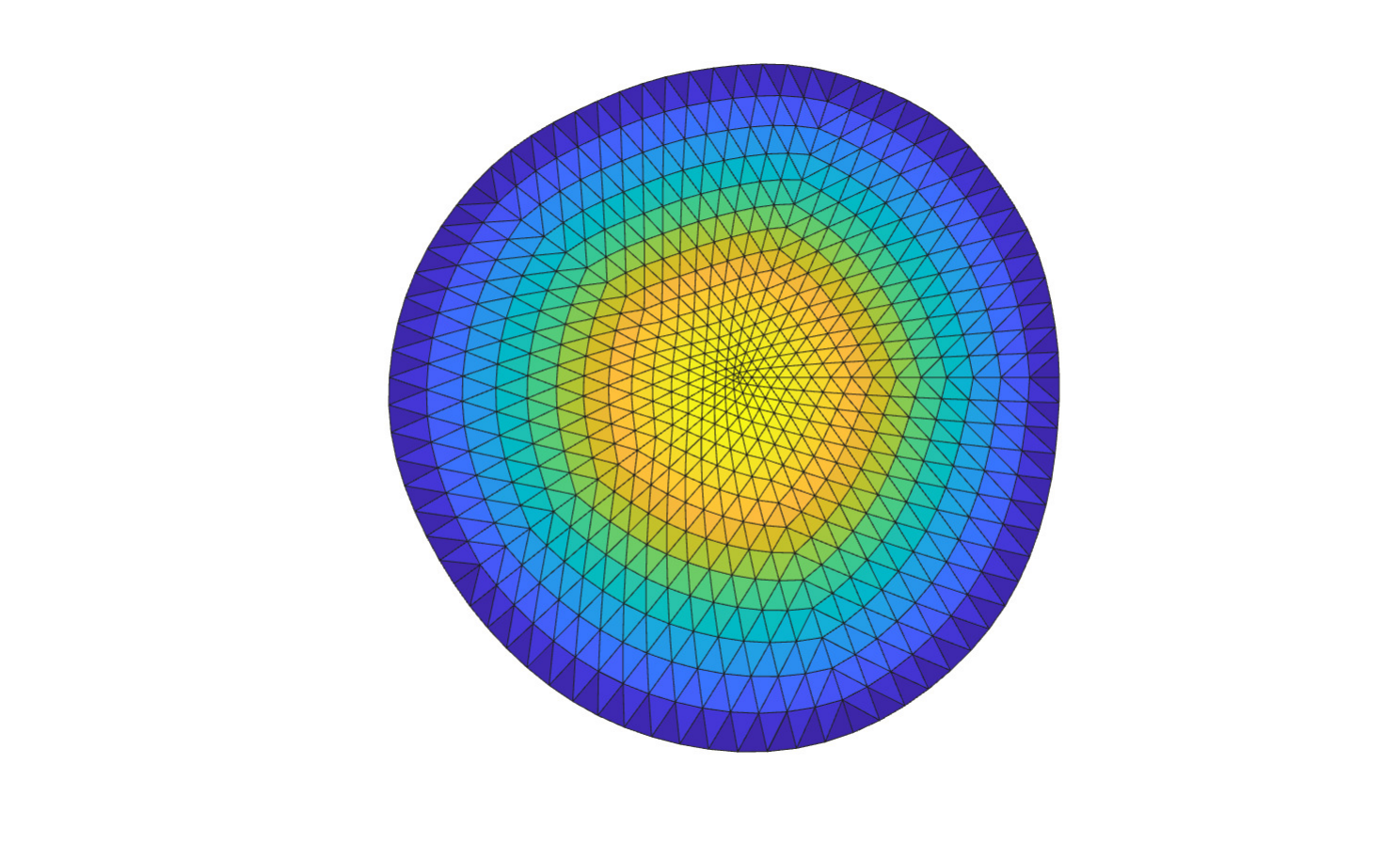}\\
		\includegraphics[width=0.45\textwidth,trim={170 40 150 30},clip]{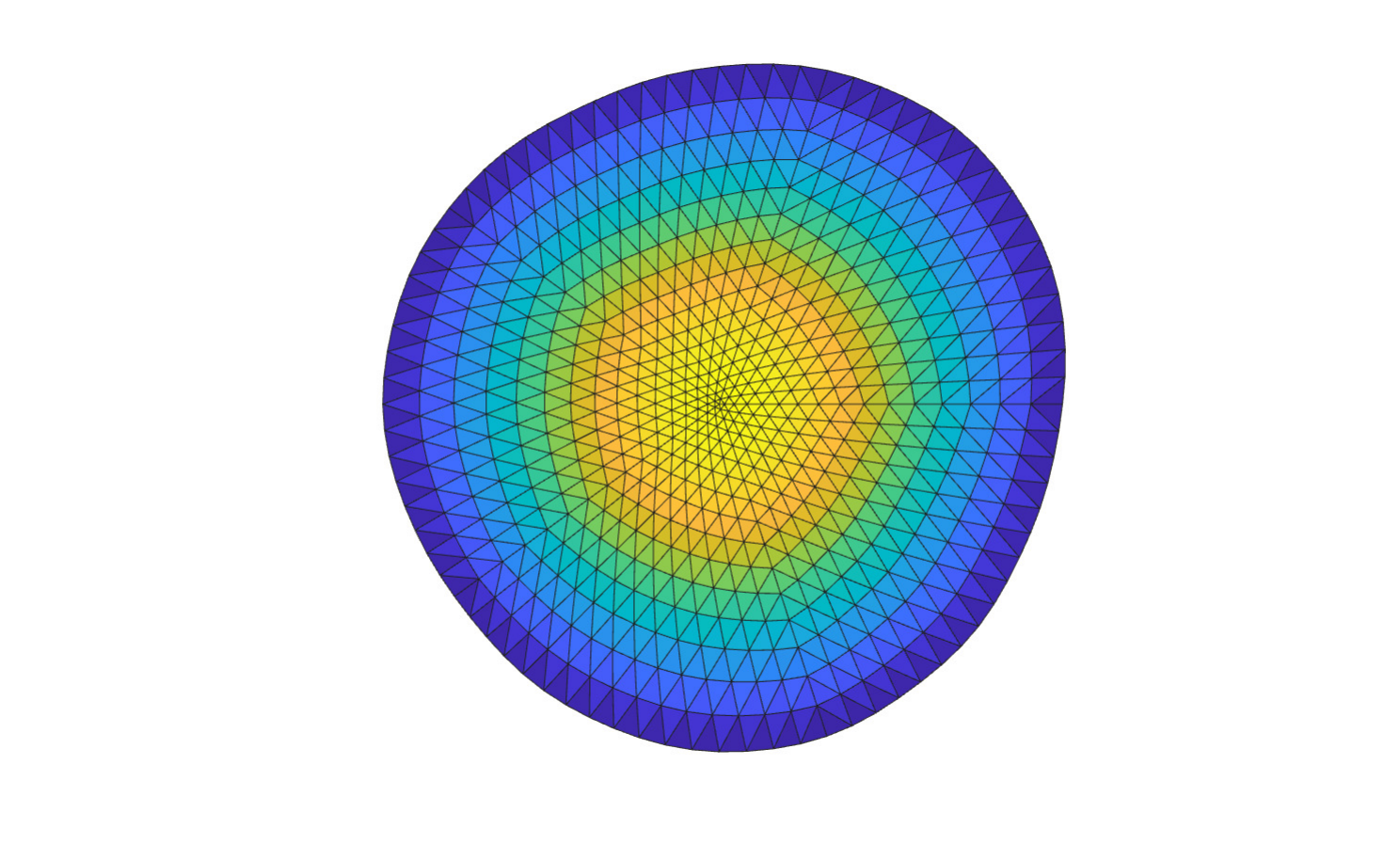}\qquad
		\includegraphics[width=0.45\textwidth,trim={170 40 150 30},clip]{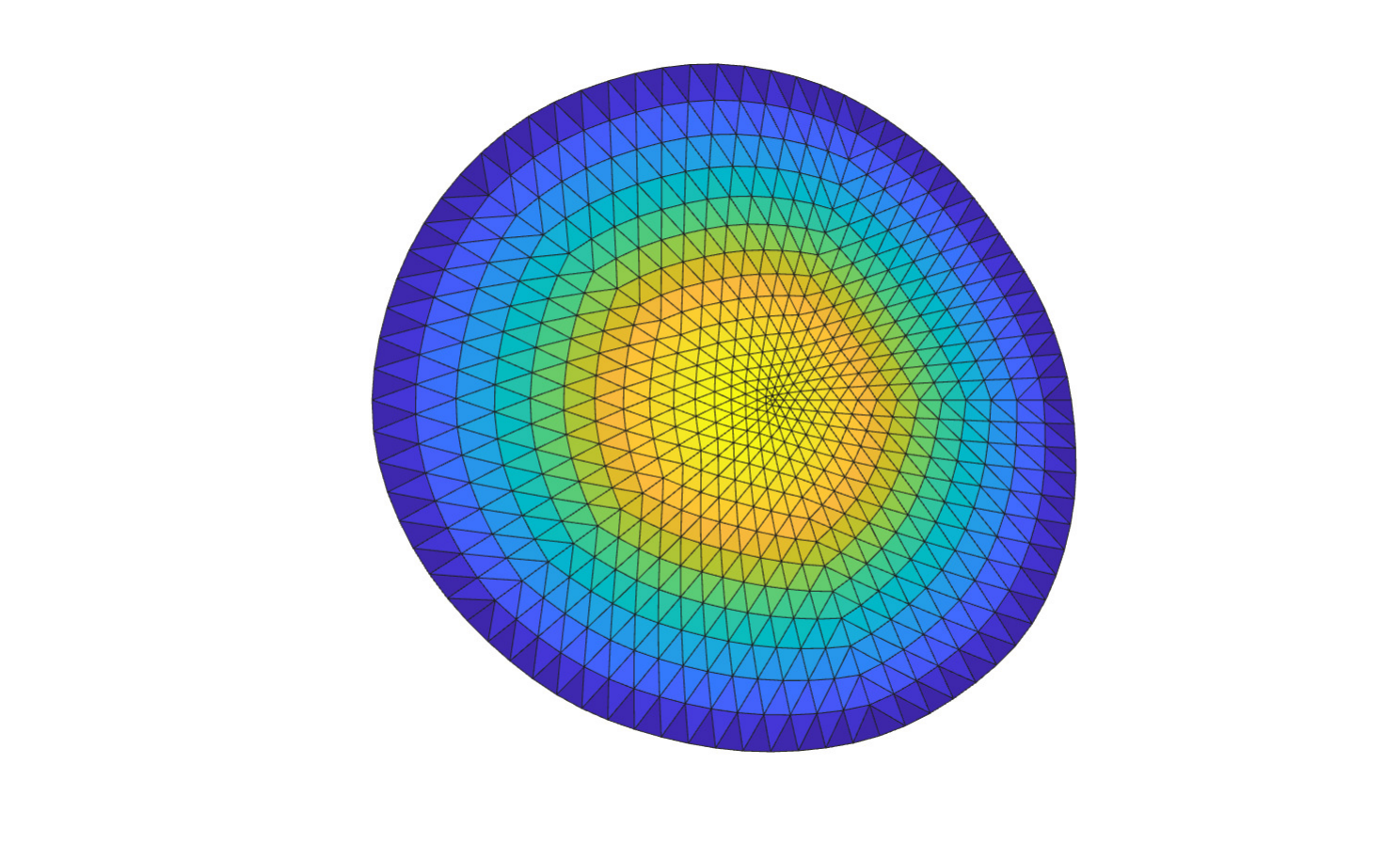}
		\begin{caption}{\label{fig:realizations}%
                                Four realizations of the random domain under consideration with 
				the associated solution of the underlying Poisson equation. 
                                Shape parametrization \eqref{a=trigonometric} with $\kappa = 0$ and \eqref{eq:lkEple}.
                                The 
				triangulation is the mapped version of the triangulation of the unit disk.
				The grading towards the origin is clearly visible.}
		\end{caption}
	\end{center}
\end{figure}

\begin{figure}[htb]
	\begin{center}
		\pgfplotsset{width=0.70\textwidth, height=0.60\textwidth}
		\begin{tikzpicture}
			\begin{loglogaxis}[grid, ymin= 0.00001, ymax = 0.1, xmin = 100, xmax = 1e6, 
				ytick={0.00001,0.0001,0.001,0.01,0.1,1},
				legend style={at={(0.98,0.90)},anchor=east},
				xtick={10,100,1000,10000,10000, 1e5,1e6}, ylabel={approximation error}, 
				xlabel={number $N$ of sampling points}]
				\addplot[line width=0.7pt,color=red,mark=*] table[x=N,y=error]{
					N error asymp1 asymp2 
					100  3.0394e-02 0.02 0.6
					200  1.7574e-02 0.01 0.3
					500  8.1733e-03 0.004 0.12
					1000  5.2421e-03 0.002 0.06
					2000  2.4062e-03 0.001 0.03
					5000  1.6035e-03 0.0004 0.012
					10000  1.1510e-03 0.0002 0.006
					20000  6.3519e-04 0.0001 0.003
					50000  2.9024e-04 0.00004 0.0012
					100000  1.3238e-04 0.00002 0.0006
					200000  9.0355e-05 0.00001 0.0003
					500000  2.6807e-05 0.000004 0.00012
					1000000  1.1338e-05 0.000002 0.00006
				};
				\addplot[line width=0.7pt,color=black,dotted] table[x=N,y=asymp1]{
					N error asymp1 asymp2 
					100  3.0394e-02 0.02 0.6
					200  1.7574e-02 0.01 0.3
					500  8.1733e-03 0.004 0.12
					1000  5.2421e-03 0.002 0.06
					2000  2.4062e-03 0.001 0.03
					5000  1.6035e-03 0.0004 0.012
					10000  1.1510e-03 0.0002 0.006
					20000  6.3519e-04 0.0001 0.003
					50000  2.9024e-04 0.00004 0.0012
					100000  1.3238e-04 0.00002 0.0006
					200000  9.0355e-05 0.00001 0.0003
					500000  2.6807e-05 0.000004 0.00012
					1000000  1.1338e-05 0.000002 0.00006
				};
				\addplot[line width=0.7pt,color=black,dotted] table[x=N,y=asymp2]{
					N error asymp1 asymp2 
					100  3.0394e-02 0.02 0.6
					200  1.7574e-02 0.01 0.3
					500  8.1733e-03 0.004 0.12
					1000  5.2421e-03 0.002 0.06
					2000  2.4062e-03 0.001 0.03
					5000  1.6035e-03 0.0004 0.012
					10000  1.1510e-03 0.0002 0.006
					20000  6.3519e-04 0.0001 0.003
					50000  2.9024e-04 0.00004 0.0012
					100000  1.3238e-04 0.00002 0.0006
					200000  9.0355e-05 0.00001 0.0003
					500000  2.6807e-05 0.000004 0.00012
					1000000  1.1338e-05 0.000002 0.00006
				};
				\addlegendentry{QMC error};
				\addlegendentry{Asymptote $N^{-1}$};
			\end{loglogaxis}
		\end{tikzpicture}
		\begin{caption}{\label{fig:error}%
				Error of the approximation to $\|\mathbb{E}[\hat{u}]\|_{H^1(D_{\refd,0})}$
				of the quasi-Monte Carlo method versus the number $N$ of sampling points.}
		\end{caption}
	\end{center}
\end{figure}

We run our experiments and measure the rate of convergence
realized by the quasi-Monte Carlo method. As we do not know the
exact solution, we compute a numerical reference solution using $2\cdot 10^6$ 
sample points which serves as reference solution. Indeed, as one 
can see in Figure~\ref{fig:error}, one observes a convergence 
rate that is nearly linear convergence \eqref{eq:QMCerror} as 
predicted from the theory in \cite{QMC}. 
The computed expectation of the random solution is shown in 
Figure~\ref{fig:expectation}. 
Note that all $2$ million domain samples are 
contained in the disk of radius $2$ centered at the origin.
\begin{figure}[hbt]
	\begin{center}
		\includegraphics[width=0.7\textwidth,trim={170 40 170 100},clip]{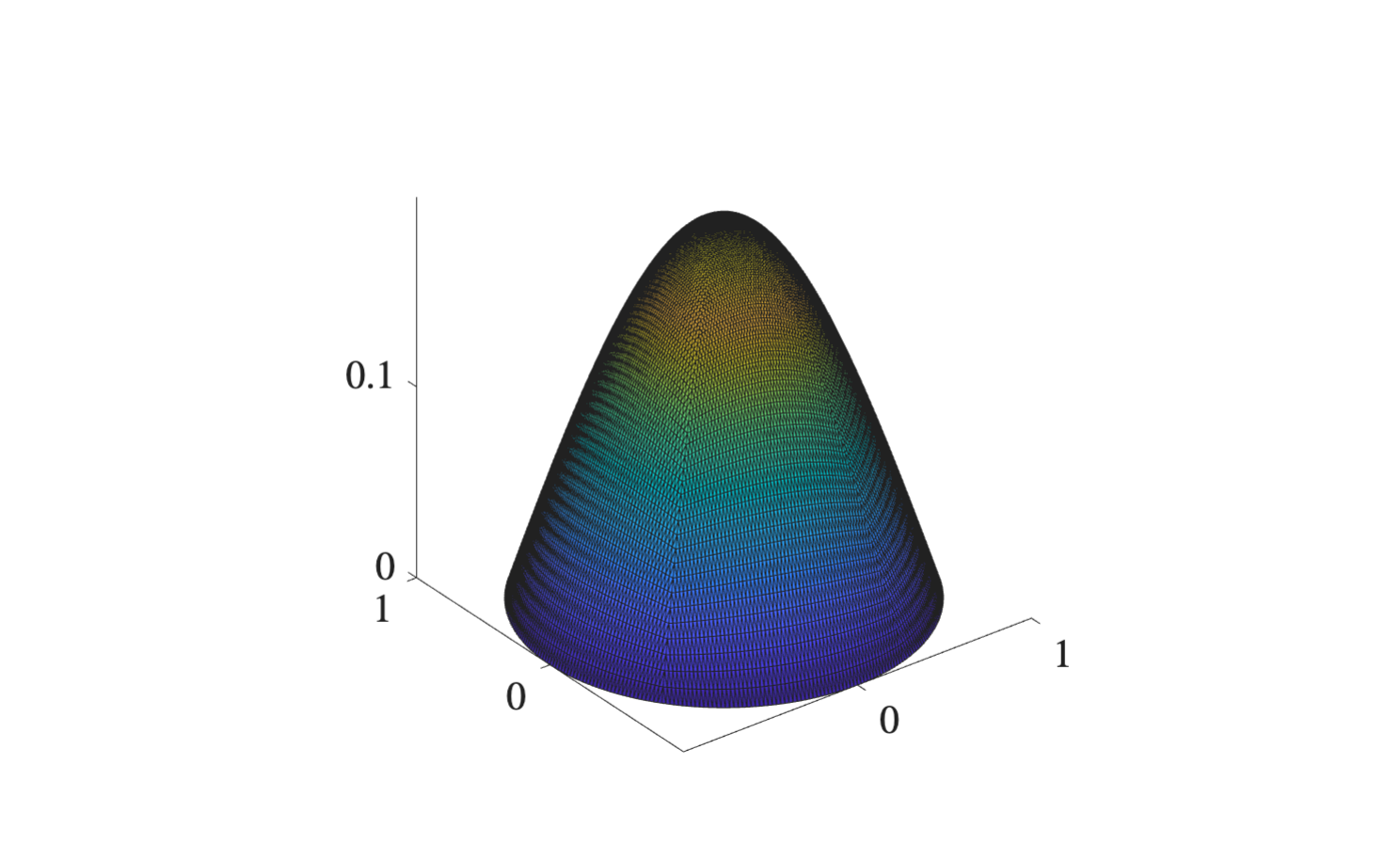}
		\begin{caption}{\label{fig:expectation}%
				The expectation of the random solution on the unit disk.}
		\end{caption}
	\end{center}
\end{figure}
\section{Conclusion}
\label{sct:conclusio}
In the present article, we studied second order elliptic boundary 
value problems on a class of random domains. The domains we
considered are homothetic with respect to a 
reference domain via a lognormally distributed
radial dilation function. 
Since the domain variations can be large 
with small probability we employed the domain mapping 
method to transform the problem posed on the random domain 
into a problem on the deterministic reference domain. 
We then have shown 
the well-posedness of the pullback problem and verified the 
analytic dependency of the random solution on the random 
input parameter by means of holomorphy arguments.

We emphasize that our particular model of the random 
domain can be straightforwardly extended. 
For example, the radial function can be modified by
\[
a(\by)(\theta) = r(\theta) 
+ \exp\Bigg(\sum_{k\in\mathbb{N}} y_k  \psi_k(\theta)\Bigg)
\]
to account for domains whose 
boundaries are homothetic, lognormal scalings 
	with respect to the radial coordinate function $r(\theta)$
	of a reference domain $D_{\refd,\kappa}$ 
	with Lipschitz boundary $\partial D_{\refd,\kappa}$, such as
	$(-1,1)^2 \backslash [-1/2,1/2]^2$.

Although the presently considered Gaussian 
geometric shape uncertainty model 
could be viewed as rather special,
it served to elucidate the mathematical and computational issues 
in numerical approximation. The present findings are expected to 
be relevant also in other contexts.
E.g.,  one can also consider the unit 
square $(0,1)^2$ as reference domain and the random domain
\[
\big\{(x_1,x_2)\in\mathbb{R}^2: 0<x_1<1,\ 
0<x_2<a(\by)(x_1)\big\}
\]
with 
\[
a(\by)(x_1) = \exp\Bigg(\sum_{k\in\mathbb{N}} y_k  \psi_k(x_1)\Bigg),
\]
where $\{\psi_k\}$ are (non-periodic) functions defined on 
the interval $(0,1)$. 
Similarly, 
the trigonometric uncertainty parametrization 
in Example~\ref{example} may be replaced by a localized one, 
e.g. of L\'{e}vy-Cieselski type.
An extension to the three dimensional setting is also possible.

We also discussed the numerical solution of the random 
boundary value problem by using finite elements in space.
We have shown that a fixed discretization suffices for the 
spatial variable. In the random parameter, we used here 
a quasi-Monte Carlo method, but also a sparse grid quadrature 
or higher order quasi-Monte Carlo methods are applicable,
see \cite{QMC2,Zung,QMC3} for example. By some numerical
experiments, we finally validated our theoretical findings.

\subsection*{Acknowledgement}
This research was funded by the Swiss National Science Foundation (SNSF) 
and the Vietnam National Foundation for Science and Technology Development 
(NAFOSTED) through the Vietnamese-Swiss Joint Research Project IZVSZ2\_229568.
It was done in part when the authors were working at the 
Vietnam Institute for Advanced Study in Mathematics (VIASM) in Hanoi, Vietnam. 
They thank VIASM for providing a fruitful research environment and working condition.
Finally, HH gratefully acknowledges the support and hospitality of the Information 
Technology Institute (ITI) of the Vietnamese National University, Hanoi. 
\bibliographystyle{plain}

\end{document}